\documentclass{amsart} 
\usepackage{graphicx,amsmath,amscd,amssymb,tensor,xfrac}
\usepackage{epstopdf}
\usepackage[nohug]{diagrams}
\usepackage{cite}
\usepackage{hyperref}
\usepackage[all,cmtip]{xy}

\usepackage{tikz,units}
\usepackage{mathtools}

\newtheorem{thm}{Theorem}[section]
\newtheorem{cor}[thm]{Corollary}
\newtheorem{lem}[thm]{Lemma}
\newtheorem{prop}[thm]{Proposition}
\newtheorem*{thm A}{Theorem A}
\newtheorem*{thm B}{Theorem B}
\newtheorem*{lem:appendix}{Lemma \ref{appendix lemma}}
\theoremstyle{definition}
\newtheorem{defn}[thm]{Definition}

\newtheorem{example}[thm]{Example}

\theoremstyle{remark}
\newtheorem{rem}[thm]{Remark}
\numberwithin{equation}{section}

\renewcommand{\b}{\mathbf}
\newcommand{\fx}[2]{\tensor[_{#1}]{\x}{_{#2}}}

\newcommand{\bb}[1]{\mathbb{#1}}
\newcommand{\D}{\mathbb D}
\newcommand{\II}{\mathbb I}
\newcommand{\JJ}{\mathbb J}
\newcommand{\R}{\mathbb R}
\newcommand{\C}{\mathbb C} 
\renewcommand{\D}{\mathcal{D}}
\newcommand{\Z}{\mathbb Z}

\newcommand{\Tc}{\mathbb T} 

\newcommand{\Lie}{\mathbb{L}\mathrm{ie}}

\newcommand{\tens}{\otimes} 
\newcommand{\dsum}{\oplus} 
\newcommand{\x}{\times}

\newcommand{\iso}{\cong} 
\newcommand{\isoto}{\overset\sim\to}
\newcommand{\toto}{\rightrightarrows}

\renewcommand{\phi}{\varphi} 
\renewcommand{\to}{\longrightarrow}
\newcommand{\lto}{\longleftarrow}
\newcommand{\oto}[1]{\overset{#1}\to}

\newcommand{\into}{\hookrightarrow}
 
\renewcommand{\mapsto}{\longmapsto}

\newcommand{\actson}{\mathbin{\,\rotatebox[origin=c]{-90}{$\circlearrowright$}\,}}
\newcommand{\ractson}{\mathbin{\,\reflectbox{\rotatebox[origin=c]{-90}{$\circlearrowright$}}\,}}
 
\newcommand{\comp}{\circ}
\newcommand{\sr}{\mathcal}

\newcommand{\Ann}{\textrm{Ann}} 
\renewcommand{\Re}{\textrm{Re}} 
\newcommand{\del}{\partial}

\newcommand{\pair}[1]{\left\langle #1 \right\rangle}

\renewcommand{\^}{\wedge} 
 
\renewcommand{\Bar}{\overline}
\renewcommand{\epsilon}{\varepsilon}
\newcommand{\wt}[1]{\widetilde{#1}}
\newcommand{\hide}[1]{} 

\newcommand{\qq}{\mathbin{
  \mathchoice{/\mkern-6mu/}
    {/\mkern-6mu/}
    {/\mkern-5mu/}
    {/\mkern-5mu/}}}
\newcommand{\Id}{\textrm{Id}}

\renewcommand{\Im}{\textrm{Im}}

\newcommand{\longversion}[1]{}

\newcommand{\note}[1]{}

\newcommand{\Inv}{\textrm{Inv}}

\newcommand{\marco}[1]{}
\newcommand{\mike}[1]{}

\title{Integration of generalized complex structures}
\author{
Michael Bailey
\and
Marco Gualtieri
}

\begin{document}
	
\begin{abstract}
We solve the integration problem for generalized complex manifolds, obtaining as the natural integrating object a \emph{weakly holomorphic symplectic groupoid}, which is a real symplectic groupoid with a compatible complex structure defined only on the associated stack, i.e., only up to Morita equivalence.  We explain how such objects \emph{differentiate} to give generalized complex manifolds, and we show that a generalized complex manifold is \emph{integrable} in this sense if and only if its underlying real Poisson structure is integrable.  
Crucial to our solution are several new technical tools which are of independent interest, namely, a reduction procedure for Lie groupoid actions on Courant algebroids, as well as certain local-to-global extension results for multiplicative forms on local Lie groupoids.
Finally, we implement our generalized complex integration procedure in several concrete examples.
\end{abstract}
	
\maketitle

\tableofcontents

\section{Introduction}

A generalized complex structure~\cite{MR2013140, Gualtieri2004,Gualtieri2011} is a differential-geometric structure on a manifold which interpolates between symplectic and complex structures. While it is defined as an integrable complex structure on an exact Courant algebroid, it can be partially described in terms of more familiar geometric structures: roughly speaking, it consists of a foliation by symplectic leaves, together with a complex structure transverse to the leaves. However, this description is misleading: the leaves may degenerate and vary in dimension across the manifold. Indeed, these are the symplectic leaves of a real Poisson structure canonically associated to the generalized complex structure, whose rank may vary. In view of this, the precise nature of the transverse complex structure requires clarification.

Since the introduction of generalized complex structures, the question of the precise relationship between a generalized complex structure and its underlying real Poisson structure has been a subject of a number of works, principally~\cite{CrainicGC} where the problem was first posed and treated in greatest detail, and in~\cite{WeinsteinCLA}, placing the problem in a larger context of complex Lie algebroid theory. In this paper we provide the first complete solution to this problem, which takes the following form.

We show that, subject to an integrability condition, a generalized complex structure with underlying real Poisson structure $P$ may be viewed, equivalently, as an extension of a \emph{Morita equivalent} real Poisson structure $P'$ to a \emph{holomorphic Poisson structure}.  The notion of Morita equivalence of Poisson manifolds was developed by Weinstein and Xu~\cite{weinstein1983,Xu1991-0}.  Roughly speaking, we say that Poisson manifolds $(M',P')$ and $(M,P)$ are Morita equivalent when there is a symplectic manifold $(E,\omega)$ defining a correspondence
\begin{align}\label{span1}
\begin{diagram}[width=2em,height=1.5em]
 & & (E,\omega) & & \\
 & \ldTo^t & & \rdTo^s & \\
 (M',P') & & & & (M,-P)
\end{diagram}
\end{align}
such that $t$ and $s$ are Poisson submersions.  The key point is that a generalized complex structure provides a holomorphic Poisson structure not on $(M,P)$, but on a Morita equivalent $(M',P')$.  Indeed, the manifold $M$ may not even admit an integrable complex structure~\cite{cavalcanti2007}.  $(M',P')$ itself is unique only up to holomorphic Morita equivalence.  Thus, we have a complex structure defined on the Morita equivalence class of $(M,P)$.  This equivalence class is a differential stack, and is a geometric model for the space of symplectic leaves of $P$.  This clarifies the nature of the transverse complex structure mentioned earlier.

To be precise about our notion of Morita equivalence, we use the language of \emph{symplectic groupoids} and \emph{integration}.  Our result may be seen as the solution to the \emph{integration} problem for generalized complex structures, in analogy with the integration of a Poisson manifold to a symplectic groupoid \cite{Coste1987}.  Given a Poisson manifold $(M,P)$, Crainic and Fernandes \cite{crainic2004} characterized the obstruction to the existence of a symplectic groupoid which \emph{differentiates} to $(M,P)$.  If $(M,P)$ is integrable in this sense, it has a unique $s$-connected and $s$-simply-connected integration, the \emph{Weinstein groupoid}.

We show that the integrating object of a generalized complex manifold $(M,\II)$ is a \emph{weakly holomorphic symplectic groupoid}, or \emph{WHSG} (see Definition \ref{WHSG defn}).  This consists of a real symplectic groupoid $G$ equipped with a weak holomorphic symplectic structure, namely, a \emph{symplectic Morita equivalence} (see Definition \ref{morita defn}) between $G$ and the imaginary part of a holomorphic symplectic groupoid $\Phi$.  
(This definition parallels, at the groupoid level, the Morita equivalence of \eqref{span1}.  The holomorphic symplectic groupoid is understood to be unique up to holomorphic symplectic Morita equivalence, hence we denote it ``weak''.)

The Morita equivalence between $G$ and $\Phi$ is a symplectic biprincipal bibundle $\Phi \actson E \ractson G$.  These data determine a \emph{generalized complex reduction} of $E$ by the action of $\Phi$ (see Section \ref{courant reduction section}), which determines a generalized complex structure on $M:=E/\Phi$ whose underlying real Poisson structure is the differentiation of $G$.  This is the \emph{differentiation} of the WHSG.  Our main result says that a generalized complex structure is the differentiation of a WHSG if and only if its underlying real Poisson structure is integrable.  We summarize this as follows:

\begin{thm}\label{summary theorem}
Let $G$ be the Weinstein symplectic groupoid of the Poisson manifold $(M,P)$.  Differentiation of weakly holomorphic symplectic groupoids defines a functor from the category of weak holomorphic symplectic structures on $G$ to the category of generalized complex structures on $M$ with underlying Poisson structure $P$. Furthermore, this functor is an essential equivalence.
\end{thm}

This result has two main components.  First, in Sections \ref{courant reduction section} and \ref{WHSG section} we develop a theory of reductions of Courant algebroids and generalized complex structures by Lie groupoid actions, and we use this to define the differentiation functor on WHSG's.  We prove functoriality of differentiation and uniqueness of the integration.

Second, in Sections \ref{localizations} to \ref{integration} we show essential equivalence, i.e., the existence of integrations. The main task is to construct a WHSG associated to any GC manifold with integrable Poisson structure, and for this the crucial ingredient is Bailey's Theorem \ref{Bailey thm}, which provides a holomorphic normal form in a neighbourhood of any point of a GC manifold.  The particular WHSG's we construct are \emph{holomorphic localizations} of the real symplectic groupoids, discussed in Section \ref{localizations}.

Section \ref{generalized complex section} is a brief introduction to Courant algebroids and generalized complex geometry, with the standard examples and local normal forms.  Section \ref{examples section} gives examples of WHSG's integrating generalized complex structures.

\begin{rem}
For simplicity, we study only structures which integrate to actual globally-defined Lie groupoids.  However, there are notions of Morita equivalence of local Lie groupoids (eg., \cite{debord2001}) which could be adapted to give a notion of \emph{local WHSG}, giving a theory of the local integration of generalized complex structures which are not integrable in the global sense.
\end{rem}

\begin{rem}(Deformation quantization)
Holomorphic symplectic groupoids model stacks with ``shifted symplectic structure.''  Such objects have been studied in \cite{Pantev2013}, and their deformation quantizations in \cite{Calaque2015}.  We expect that the deformation quantization, in this sense, of a WHSG captures an essential aspect of the ``deformation quantization'' (heretofore undefined) of the underlying generalized complex structure.
\end{rem}

\begin{rem}(Generalized complex branes)
The results of this paper provide a way to view \emph{generalized complex branes} 
from a holomorphic point of view.  In a forthcoming paper, we use these techniques to show that generalized complex branes correspond precisely to holomorphic Lagrangians in WHSG's.
\end{rem}

\noindent{\bf Acknowledgments:}
We would like to thank Henrique Bursztyn, Alejandro Cabrera, Marius Crainic, Ezra Getzler, Brent Pym and Alan Weinstein for helpful discussions. This research was supported by an NSERC Discovery Grant.

\section{Courant algebroids and generalized complex structures}\label{generalized complex section}

In this section, we introduce exact Courant algebroids and generalized complex structures, and describe their basic properties.  Most of the material here, and a more thorough introduction to generalized complex geometry, may be found in \cite{Gualtieri2011}, with the exception of the local structure theorem, which is from \cite{bailey2013}.

\begin{defn}
An \emph{exact Courant algebroid} on a manifold $M$ consists of a (real) vector bundle $\bb{E} \to M$, an \emph{anchor map} $\rho : \bb{E} \to TM$, a nondegenerate symmetric blinear pairing $\pair{\cdot,\cdot} : \bb{E} \dsum \bb{E} \to \R \x M$, and a bilinear \emph{Courant bracket} on the space of sections, such that, for the identification $\bb{E} \iso \bb{E}^*$ determined by $\pair{\cdot,\cdot}$, the following sequence is exact,
\begin{align}\label{exact sequence}
0 \to T^*M \oto{\rho^*} \bb{E} \oto{\rho} TM \to 0,
\end{align}
and, for sections $u,v,w \in \Gamma(\bb{E})$ and function $f \in C^\infty(M)$,
\begin{enumerate}
\item $[u,[v,e_3]] = [[u,v],e_3] + [v,[u,e_3]]$,
\item $[u, f v] = (rho(u) \cdot f) v + f [u,v]$,
\item $[u,u] = \tfrac{1}{2} \rho^*d\pair{u,u}$ and
\item $\rho(u) \cdot \pair{v,v} = 2\pair{[u,v],v}$.
\end{enumerate}
In an abuse of notation, we treat $\rho^* : T^*M \into \bb{E}$ as an inclusion $T^*M \subset \bb{E}$.
\end{defn}

Note that the bracket is not skew-symmetric, so that this is not a Lie algebroid.  One can always choose an isotropic splitting, $\nabla : TM \to \bb{E}$, of the sequence \eqref{exact sequence}, which realizes $\bb{E}$ as isomorphic to $TM \dsum T^*M$, with the standard symmetric pairing
\begin{align}
\pair{X + \xi, Y + \eta} = \tfrac{1}{2}(\xi(Y) + \eta(X))
\end{align}
(for $X,Y$ vectors and $\xi,\eta$ covectors), and bracket
\begin{align}\label{bracket formula}
[X+\xi,Y+\eta] = [X,Y]_{Lie} + \Lie_X \eta - \iota_Y d\xi + \iota_Y\iota_X H,
\end{align}
with $H$---the twist---a closed 3-form determined by $\nabla$ as $\iota_Y\iota_X H = 2\nabla^*[\nabla(X),\nabla(Y)]$.  We say that $\nabla$ is a \emph{flat} or \emph{involutive} splitting if $\nabla(TM)$ is closed under the bracket, i.e., if $H=0$.

\begin{defn}
For any closed 3-form we have such a Courant algebroid structure on $TM\dsum T^*M$, and when $H=0$ we call it the \emph{standard Courant algebroid} $\Tc M$.
\end{defn}

\begin{defn}
An isomorphism of Courant algebroids is just a vector bundle isomorphism which respects the structure $\rho$, $\pair{\cdot,\cdot}$ and $[\cdot,\cdot]$.

For $B$ a closed 2-form, the \emph{$B$-field-transform} (or just $B$-transform) $e^B : \bb{E} \to \bb{E}$ is given by
\begin{align}
e^B\, u := u + B(\rho(u)),
\end{align}
which may be expressed in a splitting as $e^B (X + \xi) := X + B(X) + \xi$.  If $B$ is closed, $e^B$ is a Courant automorphism.

A diffeomorphism $\phi : M \to M$ acts on $TM \dsum T^*M$ by pushforward on $TM$ and inverse pullback on $T^*M$, which we denote $\phi_*$.  $\phi_*$ is a Courant automorphism for the standard Courant algebroid.  For a twist $H$, we have that $\phi_* \comp e^B$ is a Courant automorphism if and only if $H - \phi_*(H) = dB$.

Given a splitting, all of the automorphisms of an exact Courant algebroid are generated by diffeomorphisms and $B$-transforms
\end{defn}

\begin{rem}
A $B$-transform corresponds to a change in a choice of splitting, and a non-closed $B$-transform shifts the bracket \eqref{bracket formula} by $\Delta H = dB$.  Any exact Courant algebroid is locally equivalent to the standard one (i.e., with $H=0$), and thus may be constructed by specifying an open cover of $M$ and a \v{C}ech 1-cocycle of closed 2-forms to determine how to glue the standard Courant algebroids on intersections.
\end{rem}

\begin{defn}
A \emph{generalized complex structure} on an exact Courant algebroid $\bb{E} \to M$ is a complex structure, $\II : \bb{E} \to \bb{E},\, \II^2 = -1$, on $\bb{E}$ which is orthogonal with respect to $\pair{\cdot,\cdot}$ and whose $+i$-eigenbundle is Courant--involutive.
\end{defn}

A generalized complex structure has an underlying real Poisson structure, with anchor map
\begin{align}\label{associated real Poisson structure}
P := \rho \comp \II \comp \rho^* \,:\, T^*M \to TM,
\end{align}
which is $B$-transform--invariant.

\begin{example}\label{intro example}
Both symplectic and complex structures may be realized as generalized complex structures in a standard way on $\Tc M$.  If $I : TM \to TM$ is an (integrable) complex structure, and if $\omega : TM \to T^*M$ is a symplectic structure, then
\begin{align}\label{complex and symplectic example}
\II_I &= \begin{pmatrix}
-I & 0 \\
0 & I^*
\end{pmatrix}
\quad\textnormal{and}\quad
\II_\omega = \begin{pmatrix}
0 & \omega^{-1} \\
-\omega & 0
\end{pmatrix}
\end{align}
are generalized complex.  Given a complex structure $I$ and a holomorphic Poisson structure with anchor map $\pi = IP + iP : T_\C^*M \to T_\C M$ ($P$ its imaginary part),
\begin{align}\label{hol Poisson example}
\II_{I,\pi} &= \begin{pmatrix}
-I & P \\
0 & I^*
\end{pmatrix}
\end{align}
is generalized complex.
\end{example}
\begin{rem}\label{splitting remark}
In fact, any generalized complex structure of the form \eqref{hol Poisson example}, i.e., with vanishing $TM \to T^*M$ component, is holomorphic Poisson.  Equivalently, an involutive, isotropic splitting of a Courant algebroid, invariant with respect to a generalized complex structure, determines a holomorphic Poisson structure.
\end{rem}

In \cite{bailey2013}, Bailey showed that
\begin{thm}\label{Bailey thm}
In a small enough neighbourhood of any point, a generalized complex structure is equivalent (up to a choice of splitting) to a product of a symplectic manifold with a holomorphic Poisson manifold whose Poisson tensor vanishes at the point in question.
\end{thm}

If the symplectic component has real dimension $0 \bmod 4$, then its Darboux coordinates are compatible with some complex structure, and it will be $B$-equivalent to some non-degenerate holomorphic Poisson structure.  Thus, under this parity assumption, both components in Theorem \ref{Bailey thm} are holomorphic, and we may say, more simply, 
\begin{cor}\label{parity corollary}
If a generalized complex structure has real Poisson structure of rank $0 \bmod 4$, then it is locally equivalent to a holomorphic Poisson structure.
\end{cor}

\section{Reduction of Courant algebroids by Lie groupoid actions}\label{courant reduction section}
In this section, we describe a general formalism for the action of a Lie groupoid on a Courant aglebroid or generalized complex manifold, generalizing the formalism for Lie groups developed in \cite{Bursztyn2005}.  The actions we consider are ``inner'' in a certain sense, so that we have the data necessary to perform \emph{reduction} of the Courant algebroid/generalized complex structure to the quotient.  All of the examples in this section are used later in the paper, but one should pay special attention to Example \ref{holomorphic symplectic action}, which corresponds to part of the \emph{weakly holomorphic symplectic groupoids} defined in Section \ref{WHSG section}, and whose reduction is the differentiation operation of Theorem \ref{summary theorem}.

We will only briefly recall standard material for Lie groupoids, Lie algebroids and Courant morphisms.  For a review of Lie groupoids and their actions, see \cite{mackenzie2005general} or \cite{MoerdijkMrcun}, and for Dirac relations and Courant morphisms, one can look at \cite{Li-BlandMeinrenken}.  The material from Section \ref{courant action section} onwards is new.

\begin{rem}
Throughout this section, we consider only left actions satisfying the left version of the action law.  All definitions and results can be converted to right actions in a straightforward way.
\end{rem}

\subsection{Lie groupoids}\label{lie groupoid section}

A Lie groupoid $G \toto M$ consists of manifolds $G$ (the \emph{arrows}) and $M$ (the \emph{objects} or the \emph{base} of $G$), together with smooth surjective submersions, $s,t : G \to M$ (the \emph{source} and \emph{target} maps) and $m : G \fx{s}{t} G \to G$ (the multiplication), an identity section $\Id : M \to G$, and an inversion $\Inv : G \to G$, satisfying the groupoid axioms.

A left action of $G$ on a manifold $E$ is given by a \emph{moment map} (in analogy \cite{mikami1988moments} with the moment maps of symplectic reduction), $s : E \to M$, along with a smooth map
\begin{align}
a : G \fx{s}{t} E \to E
\end{align}
satisfying the usual left action law.  We often denote $m(g_1,g_2)$ and $a(g,e)$ by $g_1\cdot g_2$ and $g\cdot e$ respectively, and a left (resp. \emph{right}) action by $G \actson E$ (resp. $E \ractson G$).

A fiber bundle $E \to B$ is a $G$--bundle if $G$ acts on $E$ preserving the fibers over $B$.  $E$ is a \emph{principal} $G$--bundle if
\begin{align}
a \tens \Id : G \fx{t}{s} E \to E \x_B E
\end{align}
is a diffeomorphism.  In other words, for any $e_1,e_2 \in E$, there is at most one $g \in G$ relating them.  Then the quotient, $E / G$, i.e., the space of orbits, is identified with the base $B$.  We note that a groupoid $G$ is itself naturally a left-- and right--principal $G$-bundle.

A multiplicative form, $\omega \in \Omega^\bullet(G)$, on a groupoid $G$ is one for which
\begin{align}
m^*(\omega) = \b{p}_1^*(\omega) + \b{p}_2^*(\omega)
\end{align}
on $G \fx{s}{t} G$.  Similarly, forms $\omega_G \in \Omega^\bullet(G)$ and $\omega_E \in \Omega^\bullet(E)$ are multiplicative for action $a$ if $\omega_G$ is multiplicative and $a^*(\omega_E) = \b{p}_G^*(\omega_G) + \b{p}_E(\omega_E)$.  A \emph{symplectic groupoid} $(G,\omega)$ is a Lie groupoid $G$ with a multiplicative symplectic form $\omega$, and a symplectic groupoid action, $G \actson E$, is a Lie groupoid action respecting symplectic structures on $G$ and $E$.

A Lie algebroid consists of a vector bundle $L \to M$, along with a Lie bracket on $\Gamma(L)$ and an \emph{anchor map} $\rho : L \to TM$, satisfying the Leibniz rule, $[X,fY] = (\rho(X)\cdot f)\, Y + f[X,Y]$.  A Lie groupoid $G$ ``differentiates'' to a Lie algebroid,
\begin{align}
\Lie(G) = T^s G|_\Id = \ker(s_*)|_\Id.
\end{align}
If $G \toto M$ is a symplectic groupoid, then $\Lie(G)$ is identified with $T^*M$, and the anchor map $\rho = t_* : \Lie(G) \to TM$ determines a Poisson structure.  In this case, $t$ is a Poisson map and $s$ is anti-Poisson.

All of the above makes sense in the holomorphic category.

\subsection{Dirac relations and Courant morphisms}\label{Dirac Courant section}

We first describe Dirac relations at the level of linear algebra, and then pass to manifolds.  If $V$ and $W$ are vector spaces with symmetric, nondegenerate bilinear forms of split signature (eg., fibers of Courant algebroids), then a \emph{Dirac relation} from $V$ to $W$ is a maximal isotropic subspace $D \subset W \x \Bar{V}$, where $\Bar{V}$ is $V$ with the opposite bilinear form.  Dirac relations $D : U \to V$ and $D' : V \to W$ can be composed \emph{as relations} in the usual way, with the resulting $D' \comp D : U \to W$ also a Dirac relation.

Given such a $V$, if $K \subset V$ is an isotropic subspace, then the \emph{reduction}
\begin{align}
V \qq K := K^\perp / K
\end{align}
also has a symmetric, nondegenerate split form.  A Dirac relation, $D : V_1 \to V_2$, has a \emph{left kernel} $K_1 = \{v \in V_1 | v \sim_D 0\}$ and a \emph{right kernel} $K_2 = \{v \in V_2 | 0 \sim_D v\}$.  While $D$ may not in general give a map $V_1 \to V_2$, it \emph{does} give an isomorphism of inner-product spaces $V_1 \qq K_1 \isoto V_2 \qq K_2$.

If $\bb{X} \to X$ is an exact Courant algebroid with anchor $\rho : \bb{X} \to TX$ and $S \subset X$ is a submanifold, then we may \emph{reduce} $\bb{X}$ to a Courant algebroid on $S$, i.e.,
\begin{align}
\bb{X}_S := \rho^{-1}(TS) / N^*S.
\end{align}
If $\nabla : TS \into \bb{X}_S$ is an isotropic, involutive splitting, then we say that $(S,\nabla)$ is a \emph{brane} supported on $S \subset X$.  $\nabla$ is equivalently determined by the maximal isotropic subbundle over $S$,
\begin{align}
\nabla(TS) + N^*S \;\subset\; \bb{X}|_S.
\end{align}
If $\bb{X}$ has a generalized complex structure $\II$, then $(S,\nabla)$ is a \emph{generalized complex brane} if this maximal isotropic subbunlde is $\II$--invariant.

If $\bb{E} \to E$ and $\bb{F} \to F$ are exact Courant algebroids, and if $p : E \to X$ and $q : F \to X$ are smooth submersions, then the \emph{Courant fiber product}, $\bb{E} \tensor[_p]{\x}{_q} \bb{F}$, is the reduction of $\bb{E} \x \bb{F}$ to $E \fx{p}{q} F \subset E \x F$, in the sense above.

If $\bb{E} \to E$ and $\bb{F} \to F$ are exact Courant algebroids, and $f : E \to F$ is a smooth map, then a \emph{Courant morphism} $\phi : \bb{E} \to \bb{F}$ covering $f$ is given by by an isotropic, involutive splitting
\begin{align}
\nabla_\phi : T(F \fx{}{f} E) \to \bb{F} \fx{}{f} \Bar{\bb{E}},
\end{align}
i.e., a brane structure on $F \fx{}{f} E \subset F \x E$, where $F \fx{}{f} E$ is the graph of $f$ and $\bb{F} \fx{}{f} \Bar{\bb{E}}$ is the corresponding Courant fiber product.  If $\bb{F}$ and $\bb{E}$ are equipped with generalized complex structures, we say that $\phi$ is a \emph{generalized complex morphism} (or \emph{generalized holomorphism}) if $\nabla_\phi$ is a generalized complex brane.  Courant morphisms (resp. generalized holomorphisms) may be composed by composing them, at each point, as Dirac relations.  Such compositions are always smooth Courant morphisms (resp. generalized holomorphisms).

\begin{rem}
If the Courant algebroids are split, i.e., if $\bb{E} = \Tc E$ and $\bb{F} = \Tc F$, then $f$ lifts to a canonical Courant morphism which respects the given splitting, and any other Courant morphism covering $f$ will be the standard one precomposed with a $B$-transform on $E$.
\end{rem}

\subsection{Lie groupoid actions on Courant algebroids}\label{courant action section}

\begin{defn}\label{courant action}
Let $G \toto M$ be a Lie groupoid, let $\bb{E} \to E$ be an exact Courant algebroid.  Then a \emph{Courant action} of $G$ on $\bb{E}$ with moment map $t : E \to M$ consists of a Courant morphism
\begin{align}
\tilde{a} : \Tc G \tensor[_s]{\x}{_t} \bb{E} \to \bb{E}
\end{align}
covering some Lie groupoid action $a : G \fx{s}{t} E \to E$, and satisfying the Courant version of the usual action law, namely,
\begin{align}\label{Courant action law}
\tilde{a} \comp (\Id_{\Tc G} \tens \tilde{a}) = \tilde{a} \comp (\tilde{m} \tens \Id_\bb{E}),
\end{align}
where $\tilde{m}$ is the standard lift of the Lie groupoid product to $\Tc G$.
\end{defn}

One could generalized this further, replacing $\Tc G$ with an arbitrary exact Courant algebroid on $G$ equipped with a multiplicative structure.  However, in such generality, we do not have enough data to do a reduction of $\bb{E}$, which is our goal.  For this, we need something like an ``inner action,'' which amounts to the additional data of a multiplicative splitting of the Courant algebroid on $G$, which means we just have $\Tc G$.  Hence Definition \ref{courant action}.

\begin{example}\label{standard Courant action}
	Any Lie groupoid action $a : G \tensor[_s]{\x}{_t} E \to E$ lifts in a canonical way to an action on the standard Courant algebroid $\Tc E$.
\end{example}

\begin{example}\label{B-transformed standard action}
	Given a Courant action $\tilde{a}$ of $G$ on $\bb{E}$, and a closed multiplicative 2-form $B$ on $G$, $B$ pulls back to $\b{p}_G^*(B)$ on $G \fx{s}{t} E$, and its $B$-transform precomposes with $\tilde{a}$ to give a new Courant action.
\end{example}

\subsection{Reduction}\label{reduction section}

Suppose $\tilde{a} : \Tc_g G \fx{s}{t} \bb{E} \to \bb{E}$ is a Courant action.  Given $g \in G$, we may view the standard splitting $s_g : T_gG \into \Tc_gG$ as a Dirac relation $s_g : 0 \to \Tc_gG$.  Then, if $e_1,e_2 \in E$ such that $g\cdot e_1 = e_2$, we have a Dirac relation
\begin{align}
\tilde{g} &: \bb{E}_{e_1} \to \bb{E}_{e_2} \notag \\
\tilde{g}&= \tilde{a}_{g,e_1} \comp (s_g \tens \Id). 
\end{align}
That is, $v_1 \sim_{\tilde{g}} v_2$ whenever there is some $u \in T_gG$ such that $(u,v_1) \sim_{\tilde{a}} v_2$.  $\tilde{g}$ has both a ``left'' and ``right'' kernel,
\begin{align}
K_{e_1}^\ell := \{v_1 \in \bb{E}_{e_1} | v_1 \sim_{\tilde{g}} 0\}
\quad \textnormal{and} \quad
K_{e_2}^r := \{v_2 \in \bb{E}_{e_2} | 0 \sim_{\tilde{g}} v_2\}.
\end{align}
As a consequence of the Courant action law, \eqref{Courant action law},
\begin{prop}
Neither $K_{e_1}^\ell$ nor $K_{e_2}^r$ depends on $g$, and $K^\ell = K^r$ as subbundles of $\bb{E}$.
\end{prop}
\note{details?}
Therefore we denote both $K^\ell$ and $K^r$ as the \emph{reduction kernel} $K \subset \bb{E}$.  Since $TG \subset \Tc G$ is isotropic and involutive and $\tilde{a}$ is a Courant morphism,
\begin{prop}
$K$ is isotropic and involutive.
\end{prop}

\subsubsection{Infinitesimal action}
A classical Lie groupoid action $a : G \tensor[_s]{\x}{_t} E \to E$ determines a Lie algebroid, $t_E^*(\Lie(G))$, on $E$, whose anchor, $da : t_E^*(\Lie(G)) \to TE$, is the \emph{infinitesimal action}.  With a Courant action $\tilde{a}$, there is a surjective map of Lie algebras
\begin{align}
d\tilde{a} : t_E^*(\Lie(G)) \to K
\end{align}
covering $da$.  $d\tilde{a}$ is given as follows.  We identify $\Lie(G)$ with $T^sG|_\Id = \ker(s_*)|_\Id$.  Suppose that $g = \Id_x \in G$, $u \in T_g^sG$ and $e \in E$.  Then, since $\tilde{a}$ is a Courant morphism, it relates $(u,0) \in T_gG \x \bb{E}_e$ to a unique $d\tilde{a}_e(u) \in \bb{E}$.

\subsubsection{Courant algebroid on the quotient}
As we explained in Section \ref{Dirac Courant section}, $\tilde{g}$ determines an isomorphism from $K^\perp_{e_1} / K_{e_1}$ to $K^\perp_{e_2} / K_{e_2}$, and thus $G$ acts on $\bb{E}_{red} := K^\perp / K$.  Then the data of a Courant algebroid pass to the quotient, and
\begin{prop}
If $G \actson E$ is principal, then $\bb{E} \qq G := \bb{E}_{red} / G = (K^\perp / K) / G$ is an exact Courant algebroid on $E/G$.
\end{prop}
This determines a Courant morphism $\bb{E} \to \bb{E} \qq G$ with kernel $K$.

\subsection{Invariant splittings}\label{invariant splittings section}

In many cases, we have a more concrete description of the reduced algebroid.  Given a Lie groupoid action, $a : G \fx{t}{s} E \to E$, a flat splitting $\nabla : TE \into \bb{E}$ determines, together with the standard splitting of $\Tc G$, a Courant lift of $a$, 
\begin{align}
\tilde{a}_\nabla : \Tc G \fx{t}{s} \bb{E} \to \bb{E},
\end{align}
as in Example \ref{standard Courant action}.  We say that $\nabla$ is \emph{invariant} with respect to this action.  In particular, this implies that the reduction kernel $K$ is contained in $\nabla(TE)$.  We note that a Courant action may admit more than one invariant splitting on $\bb{E}$, i.e., more than one splitting may determine the same $\tilde{a}$.

$\nabla$ determines an equivalence $\bb{E} \iso \Tc E$, and this identifies the reduction $\bb{E}\qq G$ with $\Tc (E / G)$.  A different invariant splitting, $\nabla'$, differs from $\nabla$ by a $B$-transform vanishing on $K$, i.e., a $B$-field basic over $E/G$, and the identifications ${\bb{E} \qq G \iso_\nabla \Tc (E/G)}$ and ${\bb{E} \qq G \iso_{\nabla'} \Tc (E/G)}$ are related by a $B$-transform by precisely this basic $B$-field.

\subsection{Generalized complex actions}

Let $\tilde{a} : \Tc G \tensor[_s]{\x}{_t} \bb{E} \to \bb{E}$ be a Courant action, with reduction kernel $K \subset \bb{E}$ as defined in Section \ref{reduction section}.

\begin{defn}\label{g hol action}
If $\bb{E}$ is endowed with a generalized complex structure $\II_E$, then $\tilde{a}$ is \emph{generalized holomorphic} if $\II_E K = K$ (so that $\II_E$ passes to a complex structure $\II_{red}$ on $\bb{E}_{red}$), and the induced $G$-action on $\bb{E}_{red}$ respects $\II_{red}$.
\end{defn}

These data and this condition are sufficient to define a reduced generalized complex structure on the reduced Courant algebroid on the quotient $E/G$.  However, in general it is hard to know whether a given Courant action is generalized holomorphic, which leads us to consider a stronger notion of generalized complex action, involving the additional data of a GC structure on $G$:

\begin{defn}\label{gc action}
A \emph{generalized complex action} of a Lie groupoid consists of a Courant groupoid action $\tilde{a} : \Tc G \tensor[_s]{\x}{_t} \bb{E} \to \bb{E}$, where $\Tc G$ and $\bb{E}$ are endowed with generalized complex structures $\II_G$ and $\II_E$, such that the multiplication of $G$ is generalized complex, the standard splitting $T G \into \Tc G$ is $\II_G$--invariant, and $\tilde{a}$ is a generalized complex morphism.
\end{defn}

The condition that $TG \subset \Tc G$ be $\II_G$--invariant is a fairly strong condition, but, as we said earlier, the splitting of $\Tc G$ is an essential part of the data of an ``inner action'' that allows us to do reduction; therefore it is natural to ask for it to be compatible with the generalized complex structure.  As per Remark \ref{splitting remark}, a generalized complex structure with an invariant splitting is actually holomorphic Poisson.  Thus, the ``most natural'' objects by which to perform reduction on a generalized complex manifold are themselves holomorphic.

\begin{prop}\label{gc action is g hol}
A generalized complex action in the sense of Definition \ref{gc action} is generalized holomorphic in the sense of Definition \ref{g hol action}, and thus determines a reduced generalized complex structure on the quotient space $E/G$.
\end{prop}

\begin{proof}
As in Section \ref{reduction section}, at points $g \in G,\, e_1,e_2=g\cdot e_1 \in E$, we have a Dirac relation
\begin{align}
\tilde{a}_{g,e_1} : \Tc_g G \x \bb{E}_{e_1} \to \bb{E}_{e_2}
\end{align}
as well as the Dirac relation $0 \to \Tc_g G$ given by the standard splitting, $\nabla$, of $\Tc G$.  By hypothesis, these relations are generalized complex, thus their composition,
\begin{align}
\tilde{g} = \tilde{a}_{g,e_1} \comp \nabla \tens \Id : \bb{E}_{e_1} \to \bb{E}_{e_2},
\end{align}
will be as well.  Thus, $K$ and $K^\perp$ are $\II$--invariant, and $\tilde{g}$ determines an $\II_{red}$--invariant isomorphism from $\bb{E}_{red,e_1}$ to $\bb{E}_{red,e_2}$.
\end{proof}

\begin{example}\label{holomorphic symplectic action}
If $(G,\Omega = B+i\omega_G)$ is a holomorphic symplectic groupoid acting on a symplectic bundle $(E,\omega_E)$, such that the action is symplectic for $\omega_G$ and $\omega_E$, then there is a natural lift of this action to a generalized complex action:

One can check that the standard lift of $a : G \tensor[_s]{\x}{_t} E \to E$ to a Courant action $\tilde{a} : \Tc G \tensor[_s]{\x}{_t} \Tc E \to \Tc E$ is a generalized complex morphism for the standard structures $\II_{\omega_G}$ and $\II_{\omega_E}$.  As-is, $TG$ is not $\II_{\omega_G}$--invariant.  Therefore, as in Example \ref{B-transformed standard action}, we modify the action and the generalized complex structure by applying the multiplicative B-transform $e^B$ to $(G,\II_{\omega_G})$, to give a modified action $\tilde{a}_B$ and generalized complex structure $\II_\Omega := e^B\II_{\omega_G} e^{-B}$.  Now we are in the holomorphic Poisson case of Example \ref{intro example}, and $TG \subset \Tc G$ \emph{is} $\II_\Omega$--invariant.  $\tilde{a}_B$ is generalized complex for $(G,\II_\Omega)$ and $(E,\II_{\omega_E})$.

If $G \actson E$ is principal, then this reduction gives a generalized complex structure $\II_B$ on $B := E/G$.  In this case, it is reasonable to call the reduction, $(\bb{E},\II_{\omega_E}) \to (\bb{E} \qq G, \II_B)$, a \emph{symplectic resolution} of $\II_B$.
\end{example}

\begin{example}\label{holomorphic on holomorphic action}
As a sub-case of Example \ref{holomorphic symplectic action}, we suppose that, not only is $G$ holomorphic symplectic, but so is $E$, with holomorphic symplectic forms $\Omega_G= B_G + i\omega_G$ and $\Omega_E = B_E + i\omega_E$ respectively, and that the action $G \actson E$ is holomorphic as well as being symplectic.

$B_E$ gives a nonstandard, $G$--invariant splitting (in the sense of Section \ref{invariant splittings section}), $S_E := e^{-B_E} TE \subset \Tc E$, for which $\II_{\omega_E} S_E = S_E$ and $K \subset S_E$.  This allows us to identify the reduction with $NK \dsum N^*K \,/\, G = \Tc (E/G)$.  Under this identification, it is easy to show:
\begin{prop}\label{hol reduction is hol poisson pushforward}
The reduction of a holomorphic symplectic manifold $E$ by the principal action of a holomorphic symplectic groupoid $G$, as described in Example \ref{holomorphic symplectic action}, is canonically isomorphic to the generalized complex structure coming from the holomorphic Poisson pushforward of $E$ to $E/G$.
\end{prop}
\end{example}

\section{Weakly holomorphic symplectic groupoids and their differentiation}\label{WHSG section}

In this section, we define the integrating objects that are the focus of this paper, and we show that they ``differentiate'' to generalized complex manifolds.  This differentiation is the functor referenced in Theorem \ref{summary theorem}.  We show that, if a generalized complex manifold has an integration, then it is unique up to topology.  It follows that the functor is 1-1 on equivalence classes.  Later, in Section \ref{integration}, we will show that it is also surjective on equivalence classes for which the real Poisson structure is integrable.

\begin{defn}\label{morita defn}
Recall \cite{MoerdijkMrcun} that a \emph{Morita equivalence} of Lie groupoids $G \toto M$ and $H \toto N$ is a \emph{biprincipal bibundle}, $G \actson E \ractson H$ between them, i.e., $G$ acts principally on $E$ from the left and $H$ acts principally from the right such that the quotient $E \to E/G = N$ is the moment map of $E \ractson H$ and the quotient $E \to E/H = M$ is the moment map of $G \actson E$, and such that the $G$ and $H$ actions commute.  The diagram to keep in mind is this:
\begin{align}
\begin{diagram}[h=2em,w=1.5em,p=.5em]
	G & \actson & E & \ractson & H \\
	\dTo <{\raisebox{2ex}{t}} \dTo >{\raisebox{2ex}{s}} & \ldTo_t & & \rdTo_s & \dTo <{\raisebox{2ex}{t}} \dTo >{\raisebox{2ex}{s}} \\
	M & & & & N
\end{diagram}
\end{align}
We say that the Morita equivalence is \emph{symplectic} if $G$ and $H$ are symplectic groupoids, $E$ is a symplectic manifold, and $G$ and $H$ act symplectically (see Section \ref{lie groupoid section}) on $E$.
\end{defn}

Morita equivalences $G \actson E \ractson G'$ and $G' \actson F \ractson G''$ may be \emph{composed} to give a Morita equivalence
\begin{align}
E \comp F = \tfrac{E \tensor[_s]{\x}{_t} F}{G'}
\end{align}
between $G$ and $G''$.  If the Morita equivalences are holomorphic and/or symplectic, so will be their composition \cite{Xu1991}.  Morita equivalences form a weak 2-category, with 2-morphisms being just equivalences of bibundles.  In this category, a Morita equivalence $G \actson E \ractson H$ may be inverted to $H \actson \Bar{E} \ractson G$ by taking the inverse actions of $G$ and $H$, and taking the inverse symplectic structure on $E$.

As we said, the integrating object we will consider is a symplectic Lie groupoid with a ``stacky complex structure'', i.e., a holomorphic structure up to Morita equivalence:
\begin{defn}\label{WHSG defn}
	A \emph{weakly--holomorphic symplectic groupoid} (or \emph{WHSG}) consists of a (real) symplectic groupoid $(G,\omega_G)$, equipped with a \emph{weak holomorphic extension}, namely, a symplectic groupoid $\Phi \toto X$ with symplectic form $\Omega = B + i\omega_\Phi$ and a (real) symplectic Morita equivalence $\Phi \actson E \ractson G$ between $(G,\omega)$ and the imaginary part, $(\Phi,\omega_\Phi)$, of $\Phi$.  We also call $\Phi$ a \emph{holomorphic atlas} for $G$.
	
	To complete the definition, we must specify the notion of equivalence for these objects.  An \emph{equivalence} between two weakly holomorphic symplectic groupoids $\Phi \actson E \ractson G$ and $\Psi \actson F \ractson H$ consists of a symplectic groupoid isomorphism between $G$ and $H$ and a \emph{holomorphic} symplectic Morita equivalence $\Phi \actson Q \ractson \Psi$ such that the resulting diagram weakly commutes (as real symplectic Morita equivalences):
	\begin{diagram}[height=2em,width=1em]
		\Phi & & & \rTo^Q  & & & \Psi \\
		& \rdTo_E & & & & \ldTo_F & \\
		& & & G \iso H & & &
	\end{diagram}
\end{defn}

\begin{defn}
We specify how a weakly holomorphic symplectic groupoid, ${\Phi \actson E \ractson G}$, \emph{differentiates} to a generalized complex structure on $M$, the base of $G$.  The action $\Phi \actson E$ is precisely the case described in Example \ref{holomorphic symplectic action}, of a holomorphic symplectic groupoid on a real symplectic manifold.  As in the example, this determines a generalized complex action of $\Phi$ on $(E,\II_{\omega_E})$, which reduces to a generalized complex structure, $\II_M$, on $M = E/\Phi$.   $(M,\II_M)$ is the \emph{derivative} and, conversely, we say that $\Phi \actson E \ractson G$ \emph{integrates} $(M,\II_M)$.
\end{defn}

Of course, we can forget the complex structure on $\Phi$ and just take $\Phi \actson E \ractson G$ as a Morita equivalence in the real symplectic category, with the quotient $E \to M$ inducing a Poisson structure on $M$.  In a symplectic Morita equivalence, $G$ integrates this Poisson structure \cite{Xu1991-0}.  Therefore:
\begin{prop}\label{only if}
If a WHSG differentiates to $(M,\II)$, then the real symplectic groupoid integrates the underlying real Poisson structure of $\II$.
\end{prop}

\subsection{Uniqueness of derivative}

Since we have a fairly weak notion of equivalence between WHSG's, we should check that differentiation to generalized complex structures respects this equivalence.  This is the functoriality referred to in Theorem \ref{summary theorem}.

\begin{prop}\label{unique derivative}
If weakly holomorphic symplectic groupoids $\Phi \actson E \ractson G$ and $\Psi \actson F \ractson G$ are equivalent via holomorphic Morita equivalence $\Phi \actson D \ractson \Psi$, then $D$ determines an isomorphism between their generalized complex derivatives, and this is functorial.
\end{prop}

\begin{proof}
Representing the bibundles as arrows, the equivalence can be represented in the following diagram:
\begin{align*}
\begin{diagram}[width=1.5em,height=1em]
\Phi & & \rTo^D & & \Psi \\
& \rdTo(2,5)_E & & \ldTo(2,5)_F & \\
& & \ldImplies(2,1)^\alpha & & \\
& & & & \\
& & & & \\
& & G & &
\end{diagram}
\end{align*}
where $D$ is a holomorphic symplectic Morita equivalence between $\Phi$ and $\Psi$, and $\alpha : F \comp D \isoto E$ is the isomorphism of symplectic bibundles which exhibits the weak commutativity of the diagram.

$\Phi *_D \Psi := \Phi \sqcup \Psi \sqcup D \sqcup \Bar{D}$ has a natural groupoid structure and, if $\Phi$ and $\Psi$ are of the same dimension, it is a holomorphic symplectic Lie groupoid.  In this case, the actions $\Phi \actson E$ and $\Psi \actson F$, along with $\alpha : F \comp D \to E$ and $\bar{\alpha} := (\alpha^{-1},\Id) : E \comp \bar{D} \to F \comp D \comp \bar{D} = F$, determine a (symplectic) action of $\Phi *_D \Psi$ on $E \sqcup F$.  This reduces to a generalized complex structure on the base of $G$, which must be the same as the reductions of each of $E$ and $F$ individually, which are thus equivalent to each other.

If $\Phi$ and $\Psi$ are not of the same dimension then, strictly speaking, $\Phi *_D \Psi$ and $E \sqcup F$ are not even manifolds, but the same relevant property will hold; namely, that the reduced Courant algebroids $\bb{E}_{red} \sqcup \bb{F}_{red}$ are $\Phi *_D \Psi$--equivariant.

If $\Phi'' \actson D' \ractson \Phi'$ and $\Phi' \actson D \ractson \Phi$ are holomorphic symplectic Morita equivalences, then functoriality of this construction follows from the associativity of $\Phi *_{D'\comp D} \Phi''$.
\end{proof}

\subsection{Uniqueness of integration}

\begin{thm}\label{unique integration}
Suppose that $G \toto M$ is a real symplectic groupoid and $\Phi \actson E \ractson G$ and $\Psi \actson F \ractson G$ are weakly holomorphic symplectic groupoids, such that their derivatives are equivalent as generalized complex structures.  Then they are equivalent as WHSG's.
\end{thm}

\begin{proof}
We have a candidate for the bibundle $\Psi \actson D \ractson \Phi$, namely, $D := E \comp \Bar{F} = (E \x_M \Bar{F}) / G$.  We already know that this is a real symplectic Morita equivalence between $\Phi$ and $\Psi$, commuting with $E$ and $F$, with symplectic form $\omega_D$ given by the reduction of $(-\omega_E,\omega_F)$.  Thus, $\Tc D$ has a generalized complex structure $\II_{\omega_D}$ coming from $\omega_D$ in the standard way.  We will construct an $\II_{\omega_D}$--invariant splitting of $\Tc D$ by pulling back, from $M$ to $E \x_M \Bar{F}$, the data which determine the generalized complex isomorphism in the hypothesis of the theorem, and then pushing forward these data to $D$.

The Courant algebroid $\Tc E \x_M \Bar{\Tc F}$ has two different reductions,
\begin{align}
\Tc D \;\overset{\tilde{G}}{\lto}\; \Tc E \x_M \Bar{\Tc F} \;\oto{r}\; (\Tc E \qq \Phi) \x_M (\Bar{\Tc F \qq \Psi}).
\end{align}
where $\tilde{G}$ is the standard Courant lift of the quotient $G : E \x_M \Bar{F} \to D$, and $r$ is a product of the Courant reductions by $\Phi$ and $\Psi$.

We claim that both of these reductions are generalized holomorphic, in the sense that they determine generalized complex branes on the graphs of the underlying maps.  (For $\tilde{G}$ it follows from the definition of $\omega_D$, and for $r$ it is true by definition of the reduced structures on $M$.)

A generalized complex isomorphism from $\Tc E \qq \Phi$ to $\Tc F \qq \Psi$ covering the identity on $M$ is precisely a generalized complex brane,
\begin{align}
\nabla : TM \into (\Tc E \qq \Phi) \x_M \Bar{(\Tc F \qq \Psi)}
\end{align}
supported on $\mathrm{Diag}(M) \subset M \x M$.  $\nabla$ is the pushforward via $r$ of a unique generalized complex brane $\nabla_{EF} := r^*(\nabla)$ supported on $E \x_M F \subset E \x M$.

The kernel, $K_r \subset \Tc E \x_M \Bar{\Tc F}$, of $r$ is the image of the infinitesimal action of $G$, thus
\begin{align}
K_r &= (\omega_E^{-1},-\omega_F^{-1}) \cdot N^*(E \x_M F) \notag \\
&= (\II_{\omega_E},-\II_{\omega_F}) \cdot N^*(E \x_M F).
\end{align}
Since $\nabla_{EF}$ is generalized complex,
\begin{align}
\nabla_{EF}\left(T(E \x_M F)\right) + N^*(E \x_M F)
\end{align}
is $(\II_{\omega_E},-\II_{\omega_F})$--invariant.  Therefore $\nabla_{EF}\left(T(E \x_M F)\right)$ contains $K_r$, and thus is contained in $K_r^\perp$.  Then $\nabla_{EF}$ passes to an $\II_{\omega_D}$--invariant splitting, $\nabla_D$, of $\Tc D \iso (K_r^\perp / K_r) / G$.  As per Remark \ref{splitting remark}, this determines a holomorphic symplectic structure on $D$.  These data are compatible with the $(\Phi,\Psi)$--bibundle structure, and in fact this construction is inverse to that of Proposition \ref{unique derivative}.
\end{proof}

Combining Proposition \ref{only if}, Theorem \ref{unique integration} and Lie's second theorem for groupoids \cite{MackenzieXu1997}:
\begin{cor}\label{s-connected uniqueness}
A generalized complex structure has at most one integrating WHSG (up to equivalence) whose real groupoid is $s$-connected and $s$-simply-connected.
\end{cor}

\section{Differentiation applied to holomorphic localizations}\label{localizations}
In this section, we describe a special class of weakly holomorphic symplectic groupoids, defined using special data, which we call \emph{holomorphic localizations}.  These will be the integrations that we construct in Section \ref{integration}.  In Section \ref{holomorphic localizations section}, and in particular in Proposition \ref{differentiating a localization}, we explain how the generalized complex structure differentiating a holomorphic localization (in the sense of Sections \ref{courant reduction section} and \ref{WHSG section}) may be computed from the special data.

First, the classical notion of a groupoid localization:
\begin{defn}
If $G \rightrightarrows M$ is a Lie groupoid and $\sr{U} = \{U_i,\ldots\}_{i\in \sr{I}}$ is an open cover of $M$, then the \emph{localization} of $G$ with respect to $\sr{U}$ is the groupoid with base
\begin{align*}
X := \bigsqcup_i U_i = \bigcup_{i, x \in U_i} (i,x)
\end{align*}
and arrows
\begin{align*}
G_\sr{U} := \bigsqcup_{i,j} G_{ij}  = \bigcup_{\substack{i,j \\ g \in G_{ij}}} (i,g,j),
\end{align*}
where $G_{ij} = s^{-1}(U_i) \cap t^{-1}(U_j)$.  The structure maps are
\begin{itemize}
\item $s(j,g,i) = (i,s(g))$
\item $t(j,g,i) = (j,t(g))$
\item $(k,h,j) \cdot (j,g,i) = (k,hg,i)$
\item $\Id_{(i,x)} = (i, \Id_x, i)$
\end{itemize}
\end{defn}

There is a natural covering map $\phi : G_\sr{U} \to G$ which collapses the disjoint union.  $\phi$ is an \emph{essential equivalence}, which in general may be used to construct a Morita equivalence.

\subsubsection{The localization bibundle}
In this case, $G_\sr{U}$ is Morita equivalent to $G$ via the bibundle
\begin{align}
E := \bigsqcup_{i} E_i = \bigcup_{\substack{i \\ g \in E_i}} (i,g),
\end{align}
where $E_i = t^{-1}(U_i) \subset G$.  $E$ has a left $G_\sr{U}$--action whose effect on indices is
\begin{align}
G_{ij} \fx{s}{t} E_i \to E_j,
\end{align}
and a right $G$--action fixing the indices:
\begin{align}
E_i \lto E_i \fx{s}{t} G.
\end{align}


\subsubsection{The \v{C}ech groupoid and special sections}\label{Cech section}
$\Id_{\sr{U}} := \phi^{-1}(\Id_M)$ is itself a subgroupoid of $G_\sr{U}$, called the \emph{\v{C}ech groupoid}, consisting of the identity elements of $G_\sr{U}$ along with additional bisections going \emph{between} the $U_i$'s.  We denote by $\Id_{ij}$ the bisection, $t^{-1}(U_j) \cap s^{-1}(U_i) \subset \Id_\sr{U}$, between $U_i$ and $U_j$.  $\Id_\sr{U}$ is just the localization of $\Id_G$, and $\phi : G_\sr{U} \to G$ is just the quotient by $\Id_\sr{U}$.

Similarly, $E$ has a special section (over $X = \bigsqcup \sr{U}$) corresponding to $\Id$ in $G$.  Denote this section $\Id_E$ (though neither $E$ nor $\Id_E$ is a groupoid), and denote the subsection $\Id_E \cap E_i$ over $U_i$ by $\Id_i$.  In fact, we have a sub-Morita-equivalence, $\Id_\sr{U} \actson \Id_E \ractson \Id_G$.

\note{picture here!!!}

\subsection{Holomorphic symplectic localizations}\label{holomorphic localizations section}

If $G$ is symplectic, then so are its localization, $G_\sr{U}$, and the bibundle, ${G_\sr{U} \actson E \ractson G}$.

\begin{defn}
A \emph{holomorphic localization} of a real symplectic groupoid, $G\toto M$, consists of a localization, $G_\sr{U}$, equipped with a multiplicative holomorphic symplectic structure whose imaginary part is equal to the symplectic structure coming from $G$.
\end{defn}

In this case, $G_\sr{U} \actson E \ractson G$ is a weakly holomorphic symplectic groupoid, and it differentiates to give a generalized complex structure $\II_M$ on $M$.  Since holomorphic localizations are precisely the integrations we construct in Section \ref{integration}, we would like to understand $\II_M$ concretely in terms of the localization data.  

\begin{rem}
To compute the generalized complex reduction of $\Tc E$ by $G_\sr{U}$, it is sufficient to look only at a subset $\D E \subset E$ which is surjective onto $M$.  In our case, we will take, as $\D E$, a neighbourhood of $\Id_E$.  We will compute $(\Tc E)_{red}|_{\D E}$, and then pass to an exact Courant algebroid on $M$ via the quotient determined by the action of $G_\sr{U}$.
\end{rem}

\begin{prop}\label{differentiating a localization}
Let $G_\sr{U}$ be a holomorphic localization of the real symplectic groupoid $G \toto M$ over the open cover $\sr{U}=\{U_1,U_2,\ldots\}$.  The weakly holomorphic symplectic groupoid $G_\sr{U} \actson E \ractson G$ determined by these data differentiates to an exact Courant algebroid, $\bb{A}$, on $M$, with a generalized complex structure, $\II$.  Then $(\bb{A},\II)$ may be computed as follows:
\begin{enumerate}
\item Each $(\bb{A},\II)|_{U_i}$ is equivalent to the standard $(\Tc U_i, \II_{I_i,\pi_i})$, where $I_i$ and $\pi_i$ are the complex structure and holomorphic Poisson structure differentiating $G_{ii}$. \label{diff local 1}
\item On $U_{ij} := U_i \cap U_j$, the copy of $\Tc U_i$ mentioned above is glued to $\Tc U_j$ by a $B$-transform $B_{ij}$, where $B_{ij}$ is the pullback to the section $\Id_{ij} \subset G_{ij}$ of the real part of the holomorphic symplectic form. \label{diff local 2}
\end{enumerate}
\end{prop}

\begin{proof}[Proof of \eqref{diff local 1}]
\renewcommand{\qedsymbol}{}
The sections $\Id_{ii} \subset G_\sr{U}$ and $\Id_i \subset E$ have neighbourhoods---call them $\D G_{ii} \subset $ and $\D E_i$ respectively---which may be identified with each other, since they are copies of the same subset of $G$.  The holomorphic symplectic structure on $\D G_{ii}$ determines a holomorphic symplectic structure, $B_i + i\omega_i$, on $\D E_i$ for which the (local) action of $\D G_{ii}$ is holomorphic symplectic.  As explained in Example \ref{holomorphic on holomorphic action}, the $B$-transformed $e^{-B_i}\cdot T(\D E_i) \subset \Tc (\D E_i)$ gives a $G_{ii}$--invariant, $\II_{\omega_E}$--invariant splitting of $\Tc (\D E_i)$ containing the reduction kernel $K$; in turn, this gives an identification of $(\Tc E \qq G_\sr{U})|_{U_i}$ with $\Tc U_i$, on which the reduced generalized complex structure is holomorphic Poisson.  This holomorphic Poisson structure, $(I_i,\pi_i)$, is precisely the derivative of the holomorphic symplectic subgroupoid $G_{ii} \toto U_i$.
\end{proof}
\begin{proof}[Proof of \eqref{diff local 2}]
$\Tc G_i$ is glued to $\Tc G_j$ by the action of $G_{ij}$, via the Courant action morphism
\begin{align}
{\tilde{a}|_{G_{ij}} : \Tc G_{ij} \fx{s}{t} \Tc G_i|_{U_{ij}} \to \Tc G_j}.
\end{align}
To get an explicit map from $\Tc G_i|_{U_{ij}}$ to $\Tc G_j$, we restrict $a$ further, to the \v{C}ech groupoid, i.e., the special section $\Id_{ij}$.
\begin{align}
\tilde{a}|_{\Id_{ij}} &: \Tc \Id_{ij} \fx{s}{t} \Tc G_i|_{U_{ij}} \to \Tc G_j \label{twisted action} \\
&: \Tc G_i|_{U_{ij}} \to \Tc  G_j \notag
\end{align}
This map reduces, on $K^\perp/K$, to the $G$-action defined in Section \ref{reduction section}, which gives the gluing for the reduction.  Recall that, for a weakly holomorphic symplectic groupoid, the Courant action $\tilde{a}$ was defined (in Example \ref{holomorphic symplectic action}) as the standard lift of the Lie groupoid action, modified by the $B$-field on $G_\sr{U}$.  The effect of \eqref{twisted action}, then, is to modify the standard lift of $\Id_{ij} : G_i|_{U_{ij}} \to G_j$ by this $B$-field pulled back to $\Id_{ij}$.
\end{proof}

\subsection{A parity condition for holomorphic localizations}\label{rank 2 case}

If a generalized complex structure is constructed, as in Proposition \ref{differentiating a localization}, by gluing holomorphic Poisson structures, then, since the complex rank of these Poisson structures is $0 \bmod 2$, the rank of the underlying \emph{real} Poisson structure will be $0 \bmod 4$.  But it is also possible for a generalized complex structure to have real Poisson rank $2 \bmod 4$, and thus the integrations of such structures cannot be realized directly by holomorphic localizations.  As explained in Section \ref{parity reduction section}, we deal with the $2 \bmod 4$ case by taking a product with symplectic $\R^2$ to give real rank $0 \bmod 4$, and noting that the result is Morita-equivalent to the original.

\section{Gauge transforms of holomorphic structures}\label{B-field section}

In this section, we describe a fundamental construction by which a complex structure on a holomorphic symplectic groupoid may be modified by a $B$-field on the base.  First, in Section \ref{holomorphic Poisson discussion}, we describe how this modification acts on a holomorphic Poisson manifold.  Then, in Section \ref{modification of groupoid}, we ``lift'' this construction to groupoids in a multiplicative way.

\subsection{Gauge transforms of holomorphic Poisson and holomorphic symplectic structures}\label{holomorphic Poisson discussion}

As explained in Section \ref{generalized complex section}, given a complex structure $I$ and a holomorphic Poisson structure $IP + iP$, we have the generalized complex structure
\begin{align}\label{P as gc}
\II_{I,P} :=
\begin{pmatrix}
-I & P \\
0 & I^*
\end{pmatrix}
\end{align}
We can ask under what conditions a $B$-field transform by some $B_{ij}$ will take $\II_{I,P}$ to another structure of the form \eqref{P as gc}, i.e., when
\begin{align}
\begin{pmatrix} 1 & 0 \\ B_{ij} & 1 \end{pmatrix}
\begin{pmatrix} -I & P \\ 0 & I^* \end{pmatrix}
\begin{pmatrix} 1 & 0 \\ -B_{ij} & 1 \end{pmatrix} &=
\begin{pmatrix} -J & P \\ 0 & J^* \end{pmatrix}
\end{align}
for some $J$.  This boils down to the equations, studied in \cite{Gualtieri2010},
\begin{align}
J  &= I + PB_{ij} \label{PB1} \\
0 &= J^* B_{ij} + B_{ij} I \label{PB2}
\end{align}
or, phrased purely in terms of $I$, $P$ and $B_{ij}$,
\begin{align}\label{poisson modification condition}
B_{ij}I + I^*B_{ij} + B_{ij}PB_{ij} = 0.
\end{align}

In the case of $P = \omega^{-1}$ for $B+i\omega$ a holomorphic symplectic structure, \emph{if} this condition is satisfied, then the generalized complex structure is transformed into one corresponding to the holomorphic symplectic structure $(B+B_{ij}) + i\omega$.

\begin{rem}\label{hol symp implies complex}
In fact, on a holomorphic symplectic manifold, the data of the complex structure $I$ and the holomorphic symplectic form $B+i\omega$ are redundant.  Any two of $I$, $B$ and $\omega$ determines the third (and on a Lie groupoid the multiplicativity of any two ensures the multiplicativity of the third).  Of course, for $B + i\omega$ to determine a \emph{bona fide} complex structure $I = \omega^{-1} B$, they must satisfy certain algebraic relations, but integrability of $I$ is given by the closedness of $B$ and $\omega$.
\end{rem}

\subsection{Modification of holomorphic symplectic groupoids}\label{modification of groupoid}

The following lemma is a reformulation of the fact \cite{Xu1991} that the $s$ and $t$ fibres of a symplectic groupoid are symplectic orthogonal.

\begin{lem}\label{symp orth}
If $\omega^{-1}$ is the nondegenerate Poisson structure on a symplectic groupoid, then
$$s_* \comp \omega^{-1} \comp t^* = t_* \comp \omega^{-1} \comp s^* = 0.$$
\end{lem}

We now describe the modification of a holomorphic symplectic structure on a groupoid:
\begin{prop}\label{modified complex groupoid}
Let $\Phi \toto X$ be a holomorphic symplectic groupoid with complex structure $I_0$ and holomorphic symplectic form $\Omega_0 = B_0 + i\omega$.  Denote also by $I_0$ the induced complex structure on $X$, and let $\pi_0 = I_0P + iP$ be the induced holomorphic Poisson structure on $X$.

If $B$ is a closed 2-form on $X$ such that
\begin{align}
BI_0 + I_0^*B + BPB = 0, \label{B cond}
\end{align}
then
\begin{align}
\Omega_1 := \Omega_0 + t^*(B) - s^*(B)
\end{align}
is a new multiplicative holomorphic symplectic structure on $\Phi$, which induces, on $X$, the  complex structure $I_1 := I_0 + PB$ and holomorphic Poisson structure $\pi_1 := I_1P + iP$.
\end{prop}

\begin{proof}
Let $B_s = s^*(B)$ and $B_t = t^*(B)$, and let $C = B_t - B_s$.  $t$ is real-Poisson and $s$ is anti-Poisson, i.e.,
\begin{align}\label{t and s}
t_*(\omega^{-1}) = P = -s_*(\omega^{-1}).
\end{align}
From Lemma \ref{symp orth} we see that
\begin{align}\label{B orth}
B_t \omega^{-1} B_s \;=\; t^* \comp B \comp t_* \comp \omega^{-1} \comp s^* \comp B \comp s_* \;=\; 0
\;=\; B_s \omega^{-1} B_t 
\end{align}
Combining \eqref{B cond}, \eqref{t and s} and \eqref{B orth},
\begin{align*}
CI_0 + I_0^*C + C\omega^{-1}C &= \left(B_tI_0 + I_0^*B_t + B_t \omega^{-1} B_t\right) \\
&  \qquad - \left(B_sI_0 + I_0^*B_s - B_s \omega^{-1} B_s\right) - B_t \omega^{-1} B_s - B_s \omega^{-1} B_t \\
&= t^*\left(BI_0 + I_0^*B + B P B\right) \\
&  \qquad - s^*\left(BI_0 + I_0^*B + B P B\right) - 0  \\
&= 0.
\end{align*}
Thus, by equation \eqref{poisson modification condition}, $\Omega_1 = \Omega_0 + C$ is indeed a holomorphic symplectic structure on $\Phi$, with new complex structure $I_1 = I_0 + \omega^{-1}C$.

Since $\Omega_0$ is multiplicative, and any 2-form expressable as $C = t^*B - s^*B$ is multiplicative \cite{bursztyn2004}, therefore $\Omega_1$ is also multiplicative (and, hence, so is $I_1$).

We claim that $I_1 = I_0 + \omega^{-1} C$ on the groupoid $\Phi$ induces $I_1 = I_0 + PB$ on the base $X$.  We check
\begin{align*}
t_* \comp I_1 &\;=\; t_* \comp I_0 \;+\; t_* \comp \omega^{-1} \comp B_t \;-\; t_* \comp \omega^{-1} \comp B_s \\
&\;=\; I_0 \comp t_* \;+\; t_* \comp \omega^{-1} \comp t^* \comp B \comp t_* \;-\; t_* \comp \omega^{-1} \comp s^* \comp B \comp s_* \\
&\;=\; I_0 \comp t_* \;+\; P \comp B \comp t_* \;-\; 0 \;=\; I_1 \comp t_*
\end{align*}
Thus, $t$ is $I_1$--holomorphic, and similarly so is $s$.  Finally, since $\pi_1$ is determined by $I_1$ and $P$, and $t$ is holomorphic and real-Poisson, therefore $t$ is holomorphic Poisson also for the new structures (similarly, $s$ is anti-Poisson).
\end{proof}

\begin{prop}\label{modification is equivalent}
The two different holomorphic symplectic groupoid structures, $\Phi$ and $\Phi'$, in Proposition \ref{modified complex groupoid} are holomorphically symplectically Morita equivalent in a natural way.
\end{prop}

\begin{proof}
Let $\sr{U} := \{X_0,X_1\}$ be the open cover of $X$ consisting of two disjoint copies of $X \iso X_i$.  $\Phi$ localizes to the (holomorphic symplectic) groupoid $\Phi_\sr{U}$ with base $Y := X_0 \sqcup X_1$.  We modify the holomorphic structure on $\Phi_\sr{U}$, as in Proposition \ref{modified complex groupoid}, by the 2-form on $Y$ which equals $0$ on $X_0$ and $B$ on $X_1$, giving us the holomorphic symplectic groupoid $\tilde{\Phi}_\sr{U}$.  $\tilde{\Phi}_\sr{U}$ contains $\Phi = \tilde{\Phi}_{00}$ and $\Phi' = \tilde{\Phi}_{11}$ as full (holomorphic symplectic) subgroupoids, but it also contains components joining them.
In particular,
$$\tilde{\Phi}_{01} := t^{-1}(X_1) \cap s^{-1}(X_0) \;\subset\; \tilde{\Phi}_\sr{U}$$
is a holomorphic symplectic, biprincipal $(\Phi',\Phi)$--bibundle, equivalent, as a real symplectic bibundle, to a copy of $\Phi$.
\end{proof}

\begin{rem}\label{modification gives derivative}
We claim that the construction in the proof of Proposition \ref{modification is equivalent} is in some sense (right-)inverse to the differentiation construction prescribed in Proposition \ref{differentiating a localization}:

The holomorphic symplectic groupoid $\tilde{\Phi}_\sr{U}$ is a \emph{holomorphic localization} of (the imaginary part of) $\Phi$.  Thus, we can differentiate it according to Proposition \ref{differentiating a localization}.  It is clear that on $X_0$ (resp. $X_1$) we get the Courant algebroid $\Tc X_0$ (resp. $\Tc X_1$) endowed with the generalized complex structure coming from $(I_0,\pi_0)$ (resp. $(I_1,\pi_1)$).  But what is the $B$-field that Proposition \ref{differentiating a localization} determines we should use to glue $\Tc X_0$ to $\Tc X_1$?

Of course, \emph{it is precisely the $B$-field on $X$ which was used to modify $\Phi$}.  To see this, we inspect the modified holomorphic symplectic form on the special section, $\Id_{01} \subset \tilde{\Phi}_{01}$.  Before the modification by $B$, this section was holomorphic-Lagrangian (since it is identified with $\Id_\Phi$).  The modification pulls back $B$ from $X_1$ and adds it to the holomorphic symplectic form.
\end{rem}

\section{Existence of integrations: From GC to WHSG}\label{integration}

Here is the main ``integration'' result:

\begin{thm}\label{main theorem}
If $(M,\II)$ is a generalized complex manifold, then there exists a weakly holomorphic symplectic groupoid which differentiates to a generalized complex structure isomorphic to $\II$, if and only if the underlying real Poisson structure integrates to a symplectic groupoid.
\end{thm}

\begin{proof}
``Only if'' is Proposition \ref{only if}.  The remaining components are throughout this section, with the main construction given in Theorem \ref{construction of holomorphic localization}.  We outline the proof here.

Because of the discussion in Section \ref{parity reduction section}, in particular Lemma \ref{parity change lemma}, we can assume without loss of generality that the real Poisson structure associated to $\II$ has rank $0 \bmod 4$.  Then, as explained in Section \ref{holomorphic covers}, it admits a \emph{holomorphic cover}.  Thus, the construction of Theorem \ref{construction of holomorphic localization} is possible.
\end{proof}

\subsection{Parity reduction}\label{parity reduction section}
We would like to reduce to the case where the real Poisson structure underlying a generalized complex manifold $(M,\II)$ has rank $0 \bmod 4$.  If instead it has rank $2 \bmod 4$, then we will try to integrate $(M,\II) \x \R^2$, where $\R^2$ has the standard symplectic structure.

The standard symplectic $\R^4$ is a groupoid over the standard symplectic $\R^2$ by the additive action.  If $G \toto M$ is a (real) symplectic groupoid, then $G \x \R^4 \toto M \x \R^2$ is also a symplectic groupoid in the obvious way, and
\begin{lem}\label{parity is equivalent}
$G \x \R^4$ is symplectically Morita equivalent to $G$ with bibundle $G \x \R^2$.
\end{lem}

$(M,P) \x (\R^2,\omega^{-1})$ is an integrable Poisson manifold precisely when $(M,P)$ is---if $G$ integrates $(M,P)$, then $G \x \R^4$ integrates $(M,P) \x (\R^2,\omega^{-1})$.  If we can find a weakly holomorphic symplectic groupoid
\begin{align*}
\Phi \actson E \ractson G \x \R^4
\end{align*}
integrating $(M,\II) \x \R^2$, then Lemma \ref{parity is equivalent} gives us another weakly holomorphic symplectic groupoid, ${\Phi \actson E \comp (G \x \R^2) \ractson G}$.

\begin{lem}\label{parity change lemma}
Let $\II$ be a generalized complex structure on exact Courant algebroid $\bb{A}$ on manifold $M$, and suppose that the derivative of the weakly holomorphic symplectic groupoid ${\Phi \actson E \ractson G \x \R^4}$ is isomorphic to $(M,\II) \x (\R^2,\II_\omega)$.  Let $E' = E \comp (G \x \R^2)$.  Then the derivative of $\Phi \actson E' \ractson G$ is isomorphic to $(M,\II)$. 
\end{lem}

\begin{proof}
For simplicity, we consider the constant sections of the groupoid $\R^4$, which are just the group $\R^2$.  Then $\R^2$ acts as a group on $E$ by diffeomorphisms.  We note here that $E' = E \comp (G \x \R^2)$ is just $E / \R^2$.

The group action $\R^2 \actson E$ lifts to an action of $\R^2$ on $\Tc E$ by pushforward and inverse pullback.  We claim that
\begin{enumerate}
\item $\R^2 \actson \Tc E$ preserves $\II_{\omega_E}$.
\item $\R^2 \actson \Tc E$ preserves $K$ (the reduction kernel of the $\Phi$ action) and $K^\perp$.
\item The action $\R^2 \actson K^\perp / K$ lifts the obvious action of $\R^2$ on $\bb{A} \x \Tc \R^2$. \label{lift claim}
\end{enumerate}

Let $S \subset TE$ be the tangents to the orbits of $\R^2$.  These orbits are symplectic and so, using the symplectic form on $E$, $\Tc S$ naturally sits inside $\Tc E$.  $\R^2 \actson \Tc E$ also preserves $\Tc S$.

Either upstairs on $\Tc E$, or downstairs on $\bb{A} \x \Tc\R^2$, we can quotient by the subbundles $\Tc S$ or $\Tc\R^2$ respectively, and then quotient by the action of $\R^2$.  Since both of these subbundles are GC-invariant, the generalized complex structure passes to these ``reductions.''

From Claim \eqref{lift claim} it follows that ``reduce by $\R^2$'' and ``reduce by $\Phi$'' commute.  In other words, the Courant action of $\Phi$ on $\Tc E$ passes to a Courant action of $\Phi$ on $(\Tc E / \Tc S) / \R^2$ (which is naturally isomorphic to $\Tc E'$), and the Courant reduction of $\Tc E'$ by this action of $\Phi$ gives the same result as the ``reduction'' $(\bb{A} \x 0_{\R^2}) / \R^2$, namely, $\bb{A}$ with generalized complex structure $\II$.

But $\Phi$ already has a Courant action on $\Tc E'$ coming from the WHSG structure.  By inspection we find that these two actions are the same.
\end{proof}

\subsection{Holomorphic covers}\label{holomorphic covers}

Let $(M,\II)$ be a generalized complex manifold whose underlying real Poisson structure, $P$, has rank $0 \bmod 4$.  Suppose, without loss of generality, that $\II$ is defined on $TM \dsum T^*M$ (with a possibly twisted bracket \eqref{bracket formula}).  Corollary \ref{parity corollary} says that, locally, $\II$ is gauge-equivalent to a holomorphic Poisson structure.  Thus, $M$ admits an open cover, $\sr{U} = \{U_i,\ldots\}$, with corresponding complex structures, $\{I_i,\ldots\}$, and $B$-fields, $\{B_i,\ldots\}$, such that
\begin{align}
\II |_{U_i} &= e^{B_i} \; \JJ_i \; e^{-B_i} \notag \\
&= \begin{pmatrix} 1 & 0 \\ B_i & 1 \end{pmatrix}
\begin{pmatrix} I_i & P|_{U_i} \\ 0 & -I_i^* \end{pmatrix}
\begin{pmatrix} 1 & 0 \\ -B_i & 1 \end{pmatrix}. \label{holomorphic cover data}
\end{align}
On each $U_i$, $P$ and $I_i$ define a holomorphic Poisson structure
\begin{align}
\pi_i = I_i P|_{U_i} + i P|_{U_i}.
\end{align}
We call the data $\{(U_i,B_i)\,,\,\ldots\}$ a \emph{holomorphic cover} of $(M,\II)$.

Given a holomorphic cover, we should not expect that the complex structures $I_i$ and $I_j$ agree on overlaps $U_{ij} = U_i \cap U_j$.  However, the holomorphic Poisson structures will be related by closed $B$-fields $B_{ij} = B_j - B_i$.  In Section \ref{holomorphic Poisson discussion}, we saw that this imposes certain relations, namely,
\begin{align}
I_i + PB_{ij} &= I_j  \qquad\textnormal{and} \label{QB1} \\
B_{ij}I_i + I_i^*B_{ij} + B_{ij}PB_{ij} &= 0. \label{QB2}
\end{align}

\subsection{The integration}\label{the integration}
Given a generalized complex structure, we will build a holomorphic localization, in the sense of Section \ref{localizations}, by taking a groupoid integrating the real Poisson structure, localizing with respect to a holomorphic cover, and then ``integrating'' the complex structures on the cover to each component of the localization in some way.

In \cite{LGSX2009} it is shown that a holomorphic Poisson structure is integrable to a holomorphic symplectic groupoid precisely when the real or imaginary part of the Poisson structure is integrable to a real symplectic groupoid.  Thus, if $(M,I,\pi)$ is a holomorphic Poisson manifold whose imaginary part has an integrating groupoid $G$, then $(I,\pi)$ has an integrating holomorphic symplectic groupoid.  It follows from Lie's second theorem for groupoids \cite{MackenzieXu1997} that we have \emph{local} uniqueness of integrations:

\begin{lem}\label{local integrations are equivalent}
If $G$ and $G'$ are symplectic groupoids, both integrating the Poisson manifold $(M,P)$, then there are neighbourhoods $U \subset G$  and $U' \subset G'$ of the identity sections such that $U \iso U'$ canonically.
\end{lem}

As a consequence,
\begin{lem}\label{local integration of complex}
If $(G,\omega)$ is a real symplectic groupoid integrating $(M,P)$, and if $(I,\pi)$ is a holomorphic Poisson structure on $M$ with $\Im(\pi)=P$, then, on a sufficiently small neighbourhood of $\Id_G$, $\omega$ is the imaginary part of a unique holomorphic symplectic structure integrating $(I,\pi)$.
\end{lem}
(An analogous claim holds generally for holomorphic Lie algebroids.)

Since, \emph{a priori}, we only have local holomorphic symplectic structures, we will need some way to extend them to the whole groupoid.  A differential form $\theta$ defined on only a subset of a Lie groupoid is \emph{multiplicative} if the relation $m^*(\theta) = \b{p}_1^*(\theta) + \b{p}_2^*(\theta)$ holds wherever it is well-defined.  Then,
\begin{lem:appendix}
Let $G$ be an $s$-connected, $s$-simply-connected Lie groupoid, let $G_\sr{U}$ be its localization with respect to an open cover $\sr{U}$, let $\D G \subset G$ be an $s$-connected neighbourhood of $\Id_G$, and let $\D_\sr{U} G \subset G_\sr{U}$ be the neigbourhood of the \v{C}ech groupoid which localizes $\D G$.  Then,
\begin{enumerate}
\item A multiplicative form on $\D_\sr{U} G$ has a unique extension to a multiplicative form on $G_\sr{U}$.
\item If this form is (holomorphic) symplectic on $\D_\sr{U} G$ then it is (holomorphic) symplectic on $G_\sr{U}$.
\end{enumerate}
\end{lem:appendix}
This is a technical result about \emph{local Lie groupoids}, which we treat in Appendix \ref{local lie groupoids section}.

We will use the essential fact that Proposition \ref{modified complex groupoid}, on the modification of a multiplicative holomorphic symplectic form, goes through without change for forms which are only locally defined.  Now we may give the main construction:

\begin{thm}\label{construction of holomorphic localization}
Let $(M,\II)$ be a generalized complex manifold that admits a holomorphic cover $\sr{U} = \{(U_i,B_i), \ldots\}$.  Let $P$ be its underlying real Poisson structure, and suppose that $(M,P)$ integrates to an $s$-connected and $s$-simply-connected symplectic groupoid $G$.  Then the localization $G_\sr{U}$ is a holomorphic symplectic groupoid which differentiates, in the sense of Section \ref{holomorphic localizations section}, to $(M,\II)$.
\end{thm} 

\begin{proof}
Denote $\Phi := G_\sr{U}$.  By Lemma \ref{local integration of complex}, the holomorphic Poisson data $(I_i,\pi_i)$ on the disjoint union $X := \bigsqcup \sr{U}$ integrates to a holomorphic symplectic structure, $\Omega_{\D \Phi}$, on a neighbourhood, $\D \Phi$, of $\Id_\Phi$.  However, for Lemma \ref{appendix lemma} to apply, we need to find a (multiplicative) holomorphic symplectic structure on a neighbourhood, $\D_\sr{U} \Phi$, of the whole \v{C}ech groupoid $\Id_\sr{U}$.

\subsubsection*{Restriction of the localization is localization of the restriction}

In the disjoint union $X$, there are two copies of $U_{ij} := U_i \cap U_j$, which we denote $U_{ij}^i \subset U_i$ and $U_{ij}^j \subset U_j$.  The restriction of the localization,
\begin{align}
\Phi^{ij} := \Phi|_{U_{ij}^i \sqcup U_{ij}^j},
\end{align}
is the same as the localization of the restriction, i.e., $G|_{U_{ij}}$ localized over the double cover $\{U_{ij}^i,U_{ij}^j\}$.

Similarly, we can get a neighbourhood, $\D^{ij} \Phi \subset \Phi^{ij}$, of the \v{C}ech groupoid restricted to $U_{ij}^i \sqcup U_{ij}^j$ by taking a neighbourhood, $\D G|_{U_{ij}}$, of $\Id_{U_{ij}}  \subset G$, and localizing over the double cover $\{U_{ij}^i,U_{ij}^j\}$.

\subsubsection*{Local holomorphic symplectic structure}

$\D G|_{U_{ij}}$ has a holomorphic symplectic form, $\Omega_i$, integrating $(I_i,\pi_i)$, and this localization construction copies $\Omega_i$ to a form $\Omega_{ij}^i$ on $\D^{ij} \Phi$.  As it stands, $\Omega_{ij}^i$ will agree with $\Omega_{\D \Phi}$ on $\D \Phi|_{U_{ij}^i}$, but not on $\D \Phi|_{U_{ij}^j}$ (where $\Omega_{\D \Phi}$ integrates $(I_j,\pi_j)$).

Therefore, we modify $\Omega_{ij}^i$, as in Propositions \ref{modified complex groupoid} and \ref{modification is equivalent}, by the $B$-field which is $0$ on $U_{ij}^i$ and $B_{ij}$ on $U_{ij}^j$.  Since this is the $B$-field which takes $(I_i,\pi_i)$ to $(I_j,\pi_j)$, and since integrations are locally unique (Lemma \ref{local integration of complex}), the modification of $\Omega_{ij}^i$ agrees with $\Omega_{\D \Phi}$ now on both $\D \Phi|_{U_{ij}^i}$ and $\D \Phi|_{U_{ij}^j}$.  Combining this construction for all pairs $(i,j)$, we get a holomorphic symplectic structure, $\Omega_{\D_\sr{U}\Phi}$, on a neighbourhood, $\D_\sr{U} \Phi$, of the \v{C}ech groupoid, whose imaginary part is just the real symplectic structure on $G_\sr{U}$.

\subsubsection*{Multiplicativity}

It remains to show that $\Omega_{\D_\sr{U}\Phi}$ is multiplicative.  To this end, we apply the analogous construction to triple intersections:

$U_{ijk} = U_i \cap U_j \cap U_k$ makes three appearances in $X$, which we call $U_{ijk}^i$, $U_{ijk}^j$ and $U_{ijk}^k$.  We denote by $\Phi^{ijk}$ the restriction of $\Phi$ to their disjoint union.  Similarly to the above, $\Phi^{ijk}$ is just the localization of $G|_{U_{ijk}}$ with respect to the triple cover.  Once again, we can integrate one of the three holomorphic Poisson structures to give a holomorphic symplectic structure on a neighbourhood, $\D^{ijk}\Phi$, of the \v{C}ech groupoid, and then adjust by a $B$-field so that this holomorphic structure matches the structure $\Omega_{\D_\sr{U}\Phi}$ constructed previously.

But Proposition \ref{modified complex groupoid}, in its local version, tells us that this modified holomorphic symplectic structure is multiplicative (wherever the multplication is defined on $\D^{ijk}\Phi$).  So $\Omega_{\D_\sr{U}\Phi}$ is multiplicative, possibly on a restriction of $\D_\sr{U}\Phi$.  Then we are in the case of Lemma \ref{appendix lemma}, and so $\Phi$ is holomorphic symplectic (with the imaginary part coming from the real symplectic structure on $G$).

\subsubsection*{Checking the derivative}

We claim that the holomorphic localization we have constructed differentiates to the original generalized complex structure.  This is a straightforward consequence of how the GC structure was presented, along with how the modified holomorphic symplectic structure determines a gluing B-field.  As we discussed in Remark \ref{modification gives derivative}, modifying the holomorphic symplectic structure on the groupoid by the pullback of $B_{ij}$ shifts the real 2-form on the special section $\Id_{ij}$ by precisely $B_{ij}$.  As in Proposition \ref{differentiating a localization}, to differentiate the holomorphic localization, we differentiate each $G_{ii}$ to give a holomorphic Poisson structure on $U_i$; then, for any $i,j$, we pull back the holomorphic symplectic form to the special section, $\Id_{ij}$, giving a $B$-field with which to glue the holomorphic Poisson structure on $U_i$ to the one on $U_j$.  This gives us precisely the $B$-field we started with, allowing us to reconstruct the original GC structure.
\end{proof}

\section{Examples}\label{examples section}

In this section we give some examples of integrations of generalized complex structures.  In Section \ref{twisted complex section}, we describe how a \emph{twisted complex structure} can be seen in the integrating holomorphic groupoid; in Section \ref{symplectic example section}, we consider the case of a symplectic manifold; finally, in Section \ref{hopf example} we consider a Hopf surface endowed with a generalized complex structure which does not admit a global holomorphic gauge.


\subsection{Twisted complex manifolds}\label{twisted complex section}
Already for generalized complex manifolds of purely complex type we can see, in some cases, a nonstandard holomorphic structure on the groupoid.

Let $(M,I)$ be a complex manifold and $h$ a real, closed $(2,1) + (1,2)$--form $h$.  This determines a \emph{twisted complex structure}, i.e., the generalized complex structure $\II = \II_{I,h}$ determined by the complex structure, but living on the twisted Courant algebroid $(\Tc M, h)$.  $\II$ is locally equivalent to $\II_I$, the standard realization of the complex structure on the untwisted $\Tc M$, i.e., $(M,\II)$ has a holomorphic cover, $\{(U_i,B_i), \,\ldots \}$, such that $\II_I|_{U_i} = e^{B_i} \II e^{-B_i}$.  The $B_i$'s will not be closed in general, but their differences, $B_{ij} := B_j - B_i$, will be.  Thus, they determine a class in $H^{1}(M,\Omega^{1,1}_{cl})$, which determines the isomorphism class of the twisted structure $\II$.

The holomorphic cover of $(M,\II)$ allows us to construct its integration via holomorphic localization.  First, we integrate the real Poisson structure, which in this case is trivial, to $G = T^*M$.  Then we localize to $G_\sr{U}$, endow it with a complex structure coming from that on $T^*G$, and then modify the complex structure on the intermediate components, $G_{ij}$, by pulling back the $B$-fields $B_{ij}$.  The point is that, while locally about $\Id$, $G_\sr{U}$ is just the usual holomorphic tangent bundle, globally it must be (holomorphically) Morita-inequivalent to $T^*M$, if $h$ is cohomologically nontrivial.

Of course, each $G_{ij}$ is equivalent as a real symplectic bibundle to (a cover of) $T^*U_{ij}$.  Not only that, but the fibers have the standard complex structure coming from $T^*M$, and the projections $s$ and $t$ are holomorphic.  Nonetheless, the data of the nontrivial gluing of $\Tc U_i$ to $\Tc U_j$ is encoded in the failure of the special section $\Id_{ij}$ to be holomorphic under the modified complex structure.

\subsection{Symplectic manifolds}\label{symplectic example section}

A symplectic groupoid $G$ integrating a symplectic manifold $(M,\omega)$ is Morita equivalent to a single point with a discrete isotropy group over it (being a quotient of $\pi_1(M)$).  Therefore, there is only one way (up to equivalence) to put a weak holomorphic structure on such a $G$.  This is not surprising since, if the complex structure is transverse to the symplectic leaves, then in this case there is nothing to say.

However, up to parity issues, one \emph{can} choose a holomorphic localization of $(M,\omega)$, and then apply our construction to give a holomorphic structure on the localization $G_\sr{U}$.  The point is that there is a wide variety of holomorphic atlases representing the trivial holomorphic structure.  These are interesting if one wants a notion of ``weakly holomorphic subobject,'' since a given subgroupoid of $G$ might be a \emph{holomorphic} subgroupoid in some, but likely not all, holomorphic atlases.  We don't go into detail here, but in upcoming work it is shown that such weakly holomorphic Lagrangian subobjects are precisely the integrations of \emph{generalized complex branes}.

Holomorphic localizations of symplectic manifolds are potentially useful for a number of other purposes, such as deformation quantization.

\subsection{Spinor formalism}\label{spinor section}
We review very briefly an alternative formalism for generalized complex geometry, which we make use of in the example of Section \ref{hopf example}.

The $+i$-eigenbundle, $L \subset \C \tens \Tc M$, of a generalized complex structure is a complex \emph{Dirac structure}, i.e., a maximal-isotropic, Courant involutive subbundle, which furthermore has \emph{real rank zero}, i.e., $L \cap \bar{L} = 0$.  Conversely, any complex Dirac structure of real rank zero determines a generalized complex structure.

$T_\C M \dsum T_\C^*M$ has a \emph{Clifford action} on the mixed degree complex forms, $\Omega_\C^\bullet(M)$, namely,
\begin{align}
(X+\xi) \cdot \rho = \iota_X \rho + \xi \^ \rho.
\end{align}
A \emph{pure spinor} is a mixed degree complex form $\rho$ for which $\Ann(\rho)$ under this action is a maximal isotropic subbundle $L \subset T_\C M \dsum T_\C^*M$.  Such maximal isotropics are in 1-1 correspondence with complex line bundles of pure spinors.  $L \cap \bar{L} = 0$ if and only if the \emph{Mukai pairing} (not discussed here; see \cite{Gualtieri2011}) of the line bundle $\kappa$ with $\bar{\kappa}$ vanishes.  Finally, $L$ is Courant involutive for some twisting 3-form $H$ if and only if
\begin{align}\label{integrability of spinors}
(d + H\^)\, \Gamma(\kappa) \in (TM \dsum T^*M) \cdot \Gamma(\kappa).
\end{align}
In particular, it is sufficient that $\kappa$ has local, $(d+H\^)$--closed generators, though this is not necessary.

Thus, generalized complex structures are in 1-1 correspondence with complex pure spinor line bundles which satisfy the Mukai/real-rank-zero condition and condition \eqref{integrability of spinors}.

A complex structure, represented as a GC structure via a spinor bundle, corresponds to its canonical bundle in the usual sense.  A symplectic structure, as in \eqref{complex and symplectic example}, is generated by the spinor
\begin{align}\label{symplectic generator}
e^{i\omega} := 1 + i\omega - \tfrac{1}{2} \omega\^\omega + \ldots
\end{align}
Given a local generator, $\zeta$, of the canonical bundle of a complex structure, a holomorphic Poisson structure $\pi$ determines a GC structure generated by the spinor
\begin{align}\label{holomorphic Poisson generator}
e^\pi \cdot \zeta = \zeta + \iota_\pi \zeta + \tfrac{1}{2} \iota_\pi \iota_\pi \zeta + \ldots.
\end{align}

A $B$-transform acts on a spinor $\rho$ by $\rho \mapsto e^B \^ \rho$.

\subsection{Hopf surface}\label{hopf example}

Let $X = \left(\C^2 \setminus \{0\}\right) / \Z$, where $1 \in \Z$ acts on $\C^2 \setminus \{0\}$ via multiplication by 2.  $X$ is a primary Hopf surface.  We will define a generalized complex structure $\II$ on $X$ by first describing it on $\C^2 \setminus \{0\}$ and then passing to the quotient.

Let $x_1,x_2$ be the standard coordinates for $\C^2$.  We take $\II$ corresponding to the pure spinor line bundle generated by the spinor,
\begin{align}\label{hopf spinor}
\rho &= 1 \,+\, \frac{1}{R^2} \left( \frac{2x_1}{\bar{x}_2} d\bar{x}_1 \^ d\bar{x}_2 \,+\, dx_1 \^ d\bar{x}_1 \,+\, dx_2 \^ d\bar{x}_2 \right) \\
&\hspace{4em} \qquad\qquad +\, \frac{1}{R^4} dx_1 \^ d\bar{x}_1 \^ dx_2 \^ d\bar{x}_2, \notag
\end{align}
where $R^2 := x_1 \bar{x}_1 + x_2 \bar{x}_2$.  This is just $e^C = 1 + C + \tfrac{1}{2}C \^ C + \ldots$, where $C$ is an almost-everywhere-defined 2-form,
\begin{align}\label{2-form for hopf}
C = \frac{1}{R^2} \left( \frac{2x_1}{\bar{x}_2} d\bar{x}_1 \^ d\bar{x}_2 \,+\, dx_1 \^ d\bar{x}_1 \,+\, dx_2 \^ d\bar{x}_2 \right).
\end{align}
Both $\rho$ and $C$ are defined whenever $\bar{x}_2 \neq 0$.  Near $\bar{x}_2 = 0$, $\II$ is given by $\rho' = \tfrac{\bar{x}_2}{x_1}$, which generates the same line bundle.  This defines an (almost) generalized complex structure on $\C^2 \setminus \{0\}$, and since $\rho$ and $C$ are homogeneous of degree $0$, they pass to the quotient $X$.

$\rho$ is $(d+H\^)$--closed for the \emph{real} 3-form
\begin{align}\label{dC}
-H = dC = \frac{1}{R^4}\Big[ (x_2 d\bar{x}_2 - \bar{x}_2 dx_2) dx_1 \^ d\bar{x}_1 \,+\, 
(x_1 d\bar{x}_1 - \bar{x}_1 dx_1) dx_2 \^ d\bar{x}_2 \Big],
\end{align}
therefore $\II$ is integrable on the $H$-twisted Courant aglebroid.

In order to construct the weakly holomorphic integration of $(X,\II)$, we need local presentations of $\II$ as holomorphic Poisson structures.  Note that $\II$ does \emph{not} admit a gauge in which it is \emph{globally} holomorphic: if it did, we would have a holomorphic Poisson structure whose vanishing locus (in our coordinates, $\{x_2=0\}$), was a connected, reduced anti-canonical divisor.  Such a divisor does not exist on any primary Hopf surface.  So our complex gauges will only be locally defined, and will determine complex structures different from the one coming from $\C^2 \setminus \{0\}$.

\begin{rem}
This generalized complex structure was first described by Gualtieri \cite{Gualtieri2014} as one part of a generalized Kahler structure.  The presentation in \cite{Gualtieri2014} is in terms of bi-Hermitian data.  To translate from those data to the spinor \eqref{hopf spinor} we have given here, one must use the bi-Hermitian--to--generalized-Kahler formula \cite[Section 6]{Gualtieri2014} , then find the $+i$-eigenbundle and thence the associated pure spinor.
\end{rem}

\subsubsection{Local form near $x_2 = 0$}\label{form near x2=0}
The imaginary part of $C$ is a symplectic form defined away from $\{x_2=0\}$, and its inverse is the underlying real Poisson structure of $\II$.  It has a coordinate expression as
\begin{align}\label{Im(C)}
2\Im(C) = C-\Bar{C} &= d\left(\log\frac{\bar{x}_1}{x_1}\right) \^ d\left(\log\frac{\bar{x}_2 x_2}{R^2}\right) 
 + d\left(\log R^2\right) \^ d\left(\log\frac{\bar{x}_2}{x_2}\right). 
\end{align}
$C$ has the same imaginary part as the 2-form
\begin{align}
W &:= d\log\left(\frac{\bar{x}_1 R^2}{x_1}\right) \^ d\log\left(\frac{\bar{x}_2}{R}\right).
\end{align}
Thus, they differ by the real B-field
\begin{align}
B_2 := W - C = \frac{x_2 \bar{x}_2}{2R^2} d\log\left(\frac{\bar{x}_1}{x_1}\right) d\log\left(\frac{\bar{x}_2}{x_2}\right).
\end{align}

We define new holomorphic coordinates
\begin{align}
w_1 = \log\left(\bar{x}_1 R^2 / x_1\right)  \quad\textnormal{and}\quad  w_2 = \bar{x}_2 / R,
\end{align}
with respect to which $W = d w_1 \^ d \log w_2$ is a (singular) holomorphic symplectic form.  The range for these coordinates is $w_1 \in \C, |w_2| < 1$, and the $\Z$-action, $(w_1,w_2) \mapsto (w_1 + 4n, w_2)$, is holomorphic (so they do in fact define a complex structure on $X_2 := X \setminus \{x_1=0\}$).  The transformation $(x_1,\bar{x}_1,x_2,\bar{x}_2) \mapsto (w_1,\bar{w}_1,w_2,\bar{w}_2)$ has Jacobian determinant $4/R^4$, which never vanishes on $\C^2 \setminus \{0\}$, and so these coordinates are nondegenerate away from $x_1 = 0$ (where they are not defined).

If we transform $\II$ by $e^{B_2}$, we have a generalized complex structure $\II_2$ on $X_2$ generated by the spinor
\begin{align*}
w_2 e^{B_2} \rho &= w_2 e^W \\
&= w_2 + dw_1 \^ dw_2,
\end{align*}
which is induced by the holomorphic Poisson structure $w_2 \del_{w_1} \^ \del_{w_2}$ (which is also $\Z$--invariant).  Note that $B_2$ is not closed, and compensates for the curvature $H$, so that $\II_2$ in this gauge is untwisted.

\begin{rem}
We do not here give details on how we found this local holomorphic symplectic structure.  Briefly, the expression \eqref{Im(C)} can be thought of as a Darboux form in singular coordinates, and it is straightforward to write a real Darboux form as the imaginary part of a holomorphic symplectic structure.
\end{rem}

\subsubsection{Local form near $x_1 = 0$}\label{form near x1=0}

The formulas near $x_1=0$ are more complicated than those given above.  We define some intermediate variables:
\begin{align}
\begin{array}{cc}
p_1 = \frac{1}{2} \left( \frac{x_1}{\bar{x}_2} + \frac{\bar{x}_1}{x_2} \right) & q_1 = \frac{1}{2} \left( \frac{x_1}{\bar{x}_2} - \frac{\bar{x}_1}{x_2} \right) \\
p_2 = x_2\bar{x}_2 & q_2 = \frac{\bar{x}_2}{x_2}
\end{array}
\end{align}
Then we define a pair of complex coordinates near $x_1=0$:
\begin{align}
z_1 = \frac{p_2\left(\sqrt{p_1^2+1} + q_1\right)}{\sqrt{p_1^2+1} - q_1}, \qquad z_2 = q_2 \left(p_1 + \sqrt{p_1^2 + 1}\right).
\end{align}
The transformation $(x_1,\bar{x}_1,x_2,\bar{x}_2) \mapsto (z_1,\bar{z}_1,z_2,\bar{z}_2)$ has vanishing Jacobian determinant along two level sets of $\Re\left(\frac{x_1}{\bar{x}_2}\right)$.  In any case, this does not occur in some tubular neighbourhood of $\{x_1=0,x_2\neq 0\}$ in $\C^2 \setminus \{0\}$, whose quotient by the $\Z$--action is some $X_1 \subset \left(X \setminus \{x_2=0\} \right)$.  The $\Z$--action is, once again, holomorphic in these coordinates, and so they determine a complex structure on $X_1$.

On $X_1$, we have the holomorphic symplectic structure $Z = d\log z_1 \^ d\log z_2$, which has the same imaginary part as $C$, and a real 2-form, $B_1 := Z - C$, relating them.  Then $\II_2 := e^{B_1}\II e^{-B_1}$ on $X_1$ is generated by the pure spinor
\begin{align}
z_1 z_2\, e^{B_1} \cdot \rho &= z_1 z_2\, e^{\frac{dz_1 \^ dz_2}{z_1z_2}} \notag \\
&= z_1 z_2 + dz_1 dz_2.
\end{align}
This corresponds to the ($\Z$--invariant) holomorphic Poisson structure ${z_1 z_2\, \del_{z_1} \del_{z_2}}$.  ($z_1$ and $z_2$ vanish nowhere in $X_1$, so in fact this is nondegenerate.)

Once again, $B_1$ is not closed, since $\II$ is twisted whereas $\II_2$ is not.  However, $B_{12} := B_2 - B_1$, relating the unwtisted structures $\II_1$ and $\II_2$, \emph{is} closed.

\subsubsection{Groupoid description}

We describe the integration in terms of the holomorphic localization.  In \cite{Gualtieri-Li}, there is a construction and classification of integrations for log-symplectic manifolds, namely, Poisson manifolds $(M^n,\pi)$ for which the section $\pi^n$ of $\^ ^n TM$ is transverse to the zero section.  Then $\pi$ is nondegenerate almost everywhere and, near its degeneracy locus, $\pi$ has a local model of the form $x\, \del_x \^ \del_y + \ldots$.

As we saw, our $X_2 \subset X$ has such a linear model---in local holomorphic coordinates.  The construction in \cite{Gualtieri-Li} works just as well for holomorphic log-symplectic structures, and even though $X$ itself is not holomorphic, the construction will still go through.

Away from $D := \{x_2 = 0\} \subset X$, the Poisson structure is nondegenerate, and so we start by considering the $s$-connected and $s$-simply-connected groupoid integrating $T^*X \iso TX$, namely, the fundamental groupoid, $\Pi_1(X)$, of paths modulo homotopies.  In general, if $\b{c} : \tilde{X} \to X$ is the universal cover of $X$, then $\Pi_1(X)$ may be identified with $\tilde{X} \x \tilde{X}$ modulo diagonal deck transformations, with $t = \b{c} \comp \b{p}_1$ and $s = \b{c} \comp \b{p}_2$.  In this case, the universal cover of $X$ is $\tilde{X} := \C^2 \setminus \{0\}$.

On $X$ (and, indeed, on $\C^2 \setminus \{0\}$), we have the real, almost-everywhere-defined symplectic structure $\omega := \Im(C)$, and the real Poisson structure $P = \omega^{-1}$.  If we pull back $\omega$ to $t^*(\omega) - s^*(\omega)$ on $\Pi_1(X)|_{X \setminus D}$, then the resulting symplectic groupoid differentiates to $P$ on $X \setminus D$.

Similarly, consider the holomorphic Poisson structures $\pi_1 = z_1 z_2\, \del_{z_1} \^ \del_{z_2}$ on $X_1$ and $\pi_2 = w_2\, \del_{w_1} \^ \del_{w_2}$ on $X_2$, with their holomorphic log-symplectic structures $Z$ and $W$ respectively.  Using the cover $\sr{X} := \{X_1,X_2\}$, we localize $\Pi_1(X)$ to
\begin{align}
\tilde{\Phi} := \Pi_1(X)_\sr{X} \toto X_1 \sqcup X_2,
\end{align}
with components $\tilde{\Phi}_{11}$, $\tilde{\Phi}_{22}$, $\tilde{\Phi}_{12}$ and $\tilde{\Phi}_{21}$.  The holomorphic symplectic form equal to $Z$ on $X_1$ and $W$ on $X_2$ pulls back through $t^* - s^*$ to give a holomorphic symplectic form $\Omega$ on $\tilde{\Phi}$---everywhere except $s^{-1}(D) \cup t^{-1}(D)$---along with an induced complex structure.  $\Omega$ differentiates to the given holomorphic Poisson structure on $(X_1 \sqcup X_2) \setminus D$.

\subsubsection{The groupoid near the singular locus}
$\Omega$ does not extend to all of $\tilde{\Phi}$, so we must change the groupoid near $D$.  This modification has been studied in, for example, \cite{Gualtieri-Li}, and from \cite{RadkoShlyakhtenko} we get the following coordinate description.

On $X_2$, we have the holomorphic coordinates $(w_1,w_2)$ defined in Section \ref{form near x1=0}, and we define a groupoid $\Phi_{22} \toto X_2$ with holomorphic coordinates $(p_1,p_2,w_1,w_2)$.  The structure maps are
\begin{align}
s(p_1,p_2,w_1,w_2) &= (w_1,w_2)  \\
  t(p_1,p_2,w_1,w_2) &= (w_1 + w_2 p_1, w_2 e^{p_2}) \notag \\
m\left((p'_1,p'_2,w'_1,w'_2),\, (p_1,p_2,w_1,w_2)\right) &= (p_1 + e^{p_2} p_1',\, p_2 + p_2',\, w_1,\, w_2) \notag \\
 \textnormal{(when the product is defined)} \notag \\ 
\textnormal{and}\quad \Inv(p_1,p_2,w_1,w_2) &= (-p_1e^{-p_2},\, -p_2,\, w_1+w_2p_1,\, w_2 e^{p_2}),  \notag
\end{align}
and the multiplicative, holomorphic symplectic form is
\begin{align}
\Omega &= t^*(dw_1 \^ d\log w_2) - s^*(dw_1 \^ d\log w_2) \notag \\
&= d(w_1 + w_2 p_1) \^ dp_2 + dp_1 \^ dw_2,
\end{align}
which extends to the singular locus $\{w_2=0\}$.

Note that this structure is equivariant for the diagonal deck transformations \note{???}, namely, the $\Z$--action $(p_1,p_2,w_1,w_2) \mapsto {(p_1,p_2,w_1 + 4n, w_2)}$.  Thus, as long as we restrict the coordinates such that $|w_2| < 1$ and $\left|w_2 e^{p_2} \right| < 1$, this does in fact define a Lie groupoid $\Phi_{22} \toto X_2$.

The groupoid homomorphism
\begin{align}
(t,s) : \Phi_{22}|_{X_2 \setminus D} \to \{X \setminus D\} \x \{X \setminus D\}
\end{align}
lifts to an injective homomorphism
\begin{align}
\Phi_{22}|_{X_2 \setminus D} \into \Pi_1(X)|_{X \setminus D}
\end{align}
by the construction of $\Pi_1(X)$ as a quotient of $\tilde{X} \x \tilde{X}$.  Thus, $\Phi_{22}$ glues to $\tilde{\Phi}_{22}$ along $\Phi_{22}|_{X_2 \setminus D}$, giving a global lie groupoid
\begin{align}
\Phi := \left(\Phi_{22} \sqcup_{X_2 \setminus D} \tilde{\Phi}_{22}\right) \,\sqcup\, \tilde{\Phi}_{11} \sqcup \tilde{\Phi}_{12} \sqcup \tilde{\Phi}_{21}.
\end{align}

$\Phi$ is a holomorphic symplectic groupoid integrating the holomorphic Poisson structure on $X_1 \sqcup X_2$.  To get the underlying real groupoid of which it is a localization, we glue $\Phi_{22}$ (non-holomorphically) to $\Pi_1(X \setminus D)$ similarly to above.

\appendix

\section{Multiplicative structures on local Lie groupoids}\label{local lie groupoids section}

In this section, we explain some general theory about the ``integration'' of multiplicative structures from local Lie groupoids to global Lie groupoids.  In the analogous, more typical setting (eg., \cite{Ortiz}) one has an $s$-connected and $s$-simply-connected Lie groupoid and some kind of multiplicative structure on its Lie algebroid (which one may view as an infinitesimal neighbourhood of the identity section), and one shows that said structure integrates to the groupoid.

Our reasons for studying this local-to-global correspondence of structures is as follows.  In the construction of Theorem \ref{main theorem}, we wish to build holomorphic localizations of a real symplectic groupoid.  However, we will not initially be able to produce a holomorphic symplectic structure on the whole of the localized groupoid.  (They will not be $s$-connected, so standard integration results don't apply.)  Instead, by using local integration results and the gauge transform of Section \ref{B-field section}, we will produce multiplicative holomorphic symplectic structures on only a local piece of the groupoid.  The results of this section, then, will allow us to extend these structures to the whole groupoid.

\begin{defn}
	A \emph{local Lie groupoid}, $\D G \toto M$, is the same as a Lie groupoid, except that multiplication need only be defined in an open set, $\D^2 G \subset \D G \x_M \D G$, containing $\Id \x_M \D G \cup \D G \x_M \Id$.  Identities and inverses (satisfying the usual conditions) should still exist in all cases, $\D^2 G$ should be $\Inv \x \Inv$--invariant, and whenever one side of the associativity equation is defined, so is the other (and they are equal).  Structure maps are smooth and $s$ and $t$ are surjective submersions.  Multiplicative structures are defined in the same way as for Lie groupoids.
\end{defn}

\subsection{Formal completions}

In a local Lie groupoid there may be pairs, $(h,g)$, which are \emph{formally composable} in the sense that $s(h) = t(g)$, but which don't lie in $\D^2 G$, thus the product $h\cdot g$ is not defined.

A local Lie groupoid, $\D G$, generates a groupoid, $\pair{\D G}$---the \emph{formal completion} of $G$---consisting of composable sequences modulo the relations $m : \D^2 G \to \D G$.  We can see this arrow space as a certain colimit.  Let
$$\D G^{(n)} = \underbrace{\D G \tensor[_t]{\x}{_s} \D G \tensor[_t]{\x}{_s} \ldots \tensor[_t]{\x}{_s} \D G}_{n \;\text{times}}.$$
For $0<i<n$, let $\widetilde{\D G}^{(n)}_i$ be the same as $\D G^{(n)}$ except with $\D^2 G$ replacing $\D G \tensor[_t]{\x}{_s} \D G$ at what would have been the $i$-th and $i+1$-th entries.  The inclusion $\D^2 G \into \D G \tensor[_t]{\x}{_s} \D G$ induces an inclusion $\widetilde{\D G}^{(n)}_i \into \D G^{(n)}$.

If it is smooth then, as a manifold, $\pair{\D G}$ is the colimit in the category of smooth manifolds of the diagram consisting of every $\D G^{(n)}$ and every $\widetilde{\D G}^{(n)}_i$, with arrows consisting of all inclusions described above, as well as \emph{local compositions},
\begin{align*}
\begin{diagram}[width=2em,height=2em]
\wt{\D G}^{(n+1)}_i & = \quad & \ldots  \tensor[_t]{\x}{_s} \D^{2} G \tensor[_t]{\x}{_s}  \ldots \\
& & \dTo^{m^n_i} \\
\D G^{(n)} & = \quad & \ldots  \tensor[_t]{\x}{_s} \D G \tensor[_t]{\x}{_s}  \ldots \\
\end{diagram}
\end{align*}
This colimit description uniquely determines $\pair{\D G}$'s smooth structure \emph{if it exists}.  In this representation, the groupoid product may be expressed as concatenation of sequences.  

\begin{rem}
	Roughly speaking, a local Lie groupoid is to its formal completion as a Lie algebroid is to its $s$-connected, $s$-simply-connected integration: elements of $\pair{\D G}$ may be factored into products of ``small'' elements (in $\D G$), uniquely up to a notion of homotopy; similarly, elements of the s.s.c. integrating groupoid may be factored into ``products of infinitesimal elements,'' i.e., Lie algebroid paths, uniquely up to homotopy.  Consequently---furthering the analogy---multiplicative structures on $\D G$ may be extended to multiplicative structures on $\pair{\D G}$, as we shall see.
\end{rem}

There is an obvious notion of \emph{local Lie groupoid homomorphism}.  Such a local homomorphism, $\phi : \D G \to \D H$, extends to a unique Lie groupoid homomorphism, $\tilde{\phi} : \pair{\D G} \to \pair{\D G}$ (if the completions are smooth).

\begin{lem}\label{local generates global}
	If $G \toto M$ is an $s$-connected and $s$-simply-connected Lie groupoid, and if $\D G \subset G$ is an $s$-connected local Lie groupoid, then $\pair{\D G} \iso G$ as Lie groupoids.
\end{lem}

\note{Picture here?}

\begin{proof}
	The inclusion $\D G \into G$ determines a smooth homomorphism $\gamma : \pair{\D G} \to G$.  Since $G$ is $s$-connected, for any $g \in G$ we may choose a path from the identity to $g$ in through its $s$-fibre.  By partitioning this path, $g$ can be factored into a sequence of ``small'' elements, each lying in $\D G$, and thus $g$ has a preimage in some $\D^{(n)} G$, showing that $\gamma$ is surjective.  Furthermore, one can find a neighbourhood of $g$ which lifts smoothly to $\D^{(n)}$ in this way.
	
	We will argue that $\gamma$ is also injective, i.e., that $g$ has a \emph{unique} lift.  Then $\pair{\D G} \iso G$ as groupoids.  By the local smoothness of the lift, it will also follow that $\pair{\D G} \iso G$ as a \emph{smooth} colimit.
	
	Suppose $g \in G$ is expressed as a product (in $G$) of elements in $\D G$ in two different ways, as $g_n \cdot g_{n-1} \cdot \ldots \cdot g_1$ and $h_m \cdot \ldots \cdot h_1$.  Since $\D G$ is $s$-connected, each $g_i$ or $h_i$ may be joined to the identity via a path in its $s$-fibre.  By composing translations (in $G$) of these paths, we get a pair of paths from $\Id_{s(g)}$ to $g$, one passing through each partial composition $g_i \cdot \ldots\cdot g_1$, and the other passing through each partial composition $h_i \cdot \ldots \cdot h_1$.  Since $G$ is $s$-simply-connected, there is a homotopy between these paths in the $s$-fibre of $g$.  We may triangulate this homotopy (including, as vertices, the partial compositions along either edge), and any such triangulation gives us a step-by-step way of converting the sequence $g_n \cdot \ldots \cdot g_1$ to $h_m \cdot \ldots \cdot h_1$ via relations \emph{in $G$}.  If our triangulation is fine enough, we can hope that each triangle corresponds to a relation in $\D G$, i.e., coming from $\D^2 G$.  We will be precise about this ``fineness'' condition.
	
	We have the map
	\begin{align}
	\tau := (\textrm{Inv},\Id) : G \tensor[_s]{\x}{_s} G \to G \tensor[_t]{\x}{_s} G,
	\end{align}
	giving us the open set $\tau^{-1}(\D^2 G)$ consisting of all pairs $(g,h)$, $s(g)=s(h)$, such that the product $h\cdot g^{-1}$ is defined in $\D G$.  Since $\Id \x_M \Id \subset \tau^{-1}(\D^2 G)$, we can find a ``square'' around the identity in $\tau^{-1}(\D^2 G)$, i.e., we can find an open neighbourhood $W \subset \D G$ of $\Id$ such that $W \tensor[_s]{\x}{_s} W \subset \tau^{-1}(\D^2 G)$.  Then for any $g,h \in W$, $h\cdot g^{-1}$ is defined in $\D G$.  In other words, for any relation $g_2 \cdot g_1 = h$ in $G$, if $g_1$ and $h$ are both in $W$, then, in $\D G$, $g_2 = h\cdot g_1^{-1}$, and so $g_2 \cdot g_1 = h$ also.
	
	All possible right-translations (in $G$) of $W$ give an open cover of the homotopy surface, and there exists a triangulation subordinate to this cover, i.e., where every triangle lies in a translation of $W$.  One can then see that every triangle is realized by a relation in $\D G$.
\end{proof}

\begin{prop}\label{local structure gives global structure}
	If $\D G \toto M$ is a local Lie groupoid with smooth formal completion $\pair{\D G}$, then a multiplicative form on $\D G$ extends to $\pair{\D G}$ multiplicatively in a unique way.
\end{prop}

\begin{proof}
	From the colimit definition, a form $F$ on $\pair{\D G}$ consists of forms $F_n$ on each $\D G^{(n)}$ which respect pullbacks through each $m^n_i$.
	
	Recall that a form is multiplicative if $m^*(F) = \b{p}_1^*(F) + \b{p}_2^*(F)$.  In $\pair{\D G}$, multiplication is concatenation of sequences, so it follows that if $F$ is to be multiplicative then, necessarily,
	\begin{align}
	F_n = \sum_{i=1}^n \b{p}_i^*(F_1).
	\end{align}
	If $F_1$ is \emph{locally} multiplicative, then we can check that, indeed, $(m^n_i)^*(F_n) = F_{n+1}$.  The form $F$ thus defined is the unique multiplicative extension of $F_1$.
\end{proof}

In the above result, we can often say something about the properties of the extension form:
\begin{lem}\label{local symplectic means global symplectic}
	If a multiplicative 2-form $F_0$ on $\D G$ is (holomorphic) symplectic then its extension to $\pair{\D G}$ is (holomorphic) symplectic.
\end{lem}
\begin{proof}
An element $g \in \pair{\D G}$ is a formal composition of elements $g_1,g_2,\ldots$ in $\D G$.  Through each $g_i$ there exist local bisections which are (holomorphic) Lagrangian for $F_0$.  The product of such a collection of bisections determines a diffeomorphism from a neighbourhood of $\Id_{s(g)}$ to a neighbourhood of $g$.  A calculation shows that, since these bisections are (holomorphic) Lagrangian and $F$ is multiplicative, this diffeomorphism preserves $F$.
\end{proof}
\longversion{
\begin{lem}\label{local symplectic means global symplectic}
If a multiplicative 2-form on a Lie groupoid is (holomorphic) symplectic near the identity then it is (holomorphic) symplectic everywhere.
\end{lem}
\begin{proof}
A closed multiplicative 2-form $F$ on a Lie groupoid $G$ admits, through any point $g \in G$, an isotropic local bisection $\sigma$.  Multiplication by $\sigma$ takes a neighbourhood of $\Id_{s(g)}$ to a neighbourhood of $g$, and since $\sigma$ is isotropic and $F$ is multiplicative, $F$ is preserved by this map.  But then any local properties of $F$ about $\Id_{s(g)}$ also hold about $g$.
\end{proof}
}

\subsection{Localized local Lie groupoids}

Given an open cover, $\sr{U}$, of $M$, the localization construction of Section \ref{localizations} may be applied without modification to a local Lie groupoid $\D G \toto M$, giving a local Lie groupoid $\D_\sr{U} G \toto \bigsqcup \sr{U}$.  We could call this a \emph{localized local Lie groupoid}.  $\D_\sr{U} G$ and its product are defined, not only in a neighbourhood of the identity, but in a neighbourhood of the \v{C}ech groupoid, $\Id_\sr{U}$.  Since a localized local Lie groupoid is just a certain kind of local Lie groupoid, it generates a formal completion as described above.

As before, quotienting by the \v{C}ech groupoid cancels the localization, so that $G = G_\sr{U} / \Id_\sr{U}$, $\D G = \D_\sr{U} G / \Id_\sr{U}$ and $\pair{\D G} = \pair{\D_\sr{U} G} / \Id_\sr{U}$.

\begin{lem}\label{localized local generates localized global}
	If $G \toto M$ is a Lie groupoid, $\D G \subset G$ is a local Lie groupoid such that $\pair{\D G} \iso G$, and $\sr{U}$ is an open cover of $M$, then
	$$\pair{\D_\sr{U} G} \iso G_\sr{U}.$$
\end{lem}

In particular, by Lemma \ref{local generates global}, this holds whenever $G$ is $s$-connected and $s$-simply-connected, and $\D G \subset G$ is $s$-connected.

\begin{proof}
	We have a canonical map $\gamma : \pair{\D_\sr{U} G} \to G_\sr{U}$ fitting into the following equivalence of short exact sequences:
	\begin{diagram}[width=2em,height=2em]
		0 & \rTo & \Id_\sr{U} & \rTo & \pair{\D_\sr{U} G} & \rTo & \pair{\D G} & \rTo & 0 \\
		& & \dTo>{\rotatebox{90}{$\sim$}} & & \dTo^\gamma & & \dTo>{\rotatebox{90}{$\sim$}} & & \\
		0 & \rTo & \Id_\sr{U} & \rTo & G_\sr{U} & \rTo & G & \rTo & 0
	\end{diagram}
	Therefore, $\gamma$ is also an isomorphism.
\end{proof}

Finally, by combining Lemma \ref{localized local generates localized global} (with the remark) and Lemma \ref{local symplectic means global symplectic}, we get the result used in Section \ref{integration}:  
\begin{lem}\label{appendix lemma}
Let $G$ be an $s$-connected, $s$-simply-connected Lie groupoid, let $G_\sr{U}$ be its localization with respect to an open cover $\sr{U}$, let $\D G \subset G$ be an $s$-connected neighbourhood of $\Id_G$, and let $\D_\sr{U} G \subset G_\sr{U}$ be the neigbourhood of the \v{C}ech groupoid which localizes $\D G$.  Then,
	\begin{enumerate}
		\item A multiplicative form on $\D_\sr{U} G$ has a unique extension to a multiplicative form on $G_\sr{U}$. \label{lem part 1}
		\item If this form is (holomorphic) symplectic on $\D G$ then it is (holomorphic) symplectic on $G$. \label{lem part 2}
	\end{enumerate}
\end{lem}

\bibliographystyle{hyperamsplain-nodash}
\bibliography{references.bib}{}

\end{document}